\documentclass[oneside,british]{amsart}
\usepackage[T1]{fontenc}
\usepackage[latin9]{inputenc}
\usepackage{amstext}
\usepackage{amsthm}
\usepackage{amssymb}

\makeatletter
\numberwithin{equation}{section}
\numberwithin{figure}{section}
\theoremstyle{plain}
\newtheorem{thm}{\protect\theoremname}[section]
\theoremstyle{definition}
\newtheorem*{defn*}{\protect\definitionname}
\theoremstyle{remark}
\newtheorem*{rem*}{\protect\remarkname}
\theoremstyle{definition}
\newtheorem{defn}[thm]{\protect\definitionname}
\theoremstyle{plain}
\newtheorem{lem}[thm]{\protect\lemmaname}
\theoremstyle{plain}
\newtheorem{prop}[thm]{\protect\propositionname}
\theoremstyle{plain}
\newtheorem{cor}[thm]{\protect\corollaryname}
\theoremstyle{plain}
\newtheorem*{thm*}{\protect\theoremname}
\theoremstyle{remark}
\newtheorem{rem}[thm]{\protect\remarkname}

\usepackage[pdftex,pdfpagelabels,bookmarks,hyperindex,hyperfigures]{hyperref}

\usepackage{thmtools}

\makeatother

\usepackage{babel}
\providecommand{\corollaryname}{Corollary}
\providecommand{\definitionname}{Definition}
\providecommand{\lemmaname}{Lemma}
\providecommand{\propositionname}{Proposition}
\providecommand{\remarkname}{Remark}
\providecommand{\theoremname}{Theorem}

\begin{document}

\title{Dimension of Bernoulli Convolutions in $\mathbb{R}^{d}$}

\author{\noindent Ariel Rapaport and Haojie Ren}



\thanks{This research was supported by the Israel Science Foundation (grant
No. 619/22). AR received support from the Horev Fellowship at the
Technion -- Israel Institute of Technology.}
\begin{abstract}
For $(\lambda_{1},...,\lambda_{d})=\lambda\in(0,1)^{d}$ with $\lambda_{1}>...>\lambda_{d}$,
denote by $\mu_{\lambda}$ the Bernoulli convolution associated to
$\lambda$. That is, $\mu_{\lambda}$ is the distribution of the random
vector $\sum_{n\ge0}\pm\left(\lambda_{1}^{n},...,\lambda_{d}^{n}\right)$,
where the $\pm$ signs are chosen independently and with equal weight.
Assuming for each $1\le j\le d$ that $\lambda_{j}$ is not a root
of a polynomial with coefficients $\pm1,0$, we prove that the dimension
of $\mu_{\lambda}$ equals $\min\left\{ \dim_{L}\mu_{\lambda},d\right\} $,
where $\dim_{L}\mu_{\lambda}$ is the Lyapunov dimension. More generally,
we obtain this result in the context of homogeneous diagonal self-affine
systems on $\mathbb{R}^{d}$ with rational translations.

The proof extends to higher dimensions the works of Breuillard and
Varjú and Varjú regarding Bernoulli convolutions on the real line.
The main novelty and contribution of the present work lies in an extension
of an entropy increase result, due to Varjú, in which the amount of
increase in entropy is given explicitly. The extension of this result
to the higher-dimensional non-conformal case requires significant
new ideas.
\end{abstract}

\maketitle

\section{Introduction}

\subsection{Main result for Bernoulli convolutions}

Let $d\ge1$ be an integer and let $(\lambda_{1},...,\lambda_{d})=\lambda\in(0,1)^{d}$
be such that $\lambda_{1}>...>\lambda_{d}$. Denote by $\mu_{\lambda}$
the distribution of the random vector $\sum_{n\ge0}\pm\left(\lambda_{1}^{n},...,\lambda_{d}^{n}\right)$,
where the $\pm$ signs are chosen independently and with equal weight.
The measure $\mu_{\lambda}$ is called the Bernoulli convolution associated
to $\lambda$. Bernoulli convolutions were studied by many authors
over the years, especially in the case $d=1$. In higher dimensions,
they are perhaps the most basic example of non-conformal stationary
fractal measures.

One of the most natural and studied questions regarding Bernoulli
convolutions is to determine their dimension. A Borel probability
measure $\theta$ on $\mathbb{R}^{d}$ is said to be exact dimensional
if there exists a number $\dim\theta$ such that
\[
\underset{\delta\downarrow0}{\lim}\frac{\log\theta(B(x,\delta))}{\log\delta}=\dim\theta\text{ for }\theta\text{-a.e. }x,
\]
where $B(x,\delta)$ is the closed ball with centre $x$ and radius
$\delta$. By the work of Feng and Hu \cite{FH-dimension}, it follows
that $\mu_{\lambda}$ is always exact dimensional.

The dimension of $\mu_{\lambda}$ has a natural upper bound. It is
called the Lyapunov dimension and is denoted by $\dim_{L}\mu_{\lambda}$.
Setting
\[
m:=\max\left\{ 0\le k\le d\::\:\Pi_{j=1}^{k}\lambda_{j}\ge1/2\right\} ,
\]
the Lyapunov dimension is defined as follows:
\[
\dim_{L}\mu_{\lambda}:=\begin{cases}
m+\frac{\log2+\sum_{j=1}^{m}\log\lambda_{j}}{-\log\lambda_{m+1}} & ,\text{ if }m<d\\
d\frac{\log2}{-\sum_{j=1}^{d}\log\lambda_{j}} & ,\text{ if }m=d
\end{cases}.
\]
It always holds that
\begin{equation}
\dim\mu_{\lambda}\le\min\left\{ \dim_{L}\mu_{\lambda},d\right\} ,\label{eq:natural ub for dim(mu_lambda)}
\end{equation}
and it is expected equality occurs in the absence of algebraic obstructions.

Suppose next that $d=1$, in which case $\lambda$ is a real number
in the interval $(0,1)$. When $0<\lambda<1/2$, it holds that $\mu_{\lambda}$
is a strongly separated self-similar measure, and so it is easy to
see that $\dim\mu_{\lambda}=\dim_{L}\mu_{\lambda}=\frac{\log2}{-\log\lambda}$.
When $1/2\le\lambda<1$ is algebraic but not a root of a nonzero polynomial
with coefficients $\pm1,0$, it follows by the work of Hochman \cite{Ho1}
that $\dim\mu_{\lambda}=1$. When $1/2<\lambda<1$ is transcendental,
the equality $\dim\mu_{\lambda}=1$ was established by Varjú \cite{Var-Bernoulli}.
By the works of Erd\H{o}s \cite{erdos1939family} and Garsia \cite{garsia_singularity},
it follows that $\dim\mu_{\lambda}<1$ whenever $1/2<\lambda<1$ is
a reciprocal of a Pisot number\footnote{Recall that a Pisot number is a positive algebraic integer all of
whose Galois conjugates are inside the open unit disk.}. It is an open problem whether the inequality $\dim\mu_{\lambda}<1$
for $1/2<\lambda<1$ implies that $\lambda^{-1}$ is Pisot.

Next, assume that $d=2$. When $\lambda_{1}<1/2$, it is again easy
to see that $\dim\mu_{\lambda}=\dim_{L}\mu_{\lambda}$. Przytycki
and Urba\'{n}ski \cite{MR1002918} considered the situation in which
$\lambda_{2}=1/2$. They established that $\dim\mu_{(\lambda_{1},1/2)}=\dim_{L}\mu_{(\lambda_{1},1/2)}$
whenever $\mu_{\lambda_{1}}$ is absolutely continuous (which, according
to Solomyak \cite{MR1356783}, holds for Lebesgue a.e. $\lambda_{1}\in(1/2,1)$).
A statement of the same spirit has been obtained in \cite{MR1301464}.
It is also shown in \cite{MR1002918} that strict inequality occurs
in (\ref{eq:natural ub for dim(mu_lambda)}) when $\lambda_{2}=1/2$
and $\lambda_{1}^{-1}$ is Pisot. In the case $\lambda_{2}>1/2$,
Shmerkin \cite{MR2269414} established that the Hausdorff dimension
of the support of $\mu_{\lambda}$ equals the Lyapunov dimension for
Lebesgue a.e. $\lambda=(\lambda_{1},\lambda_{2})$ with $\lambda_{1}\lambda_{2}<1/2<\lambda_{2}$.

For general $d\ge1$, in \cite{Rap_SA_diag} the first author extended
Hochman's work \cite{Ho1} and proved that equality holds in (\ref{eq:natural ub for dim(mu_lambda)})
whenever $\lambda_{1},...,\lambda_{d}$ are algebraic but not roots
of polynomials with coefficients $\pm1,0$. By extending Varjú's work
\cite{Var-Bernoulli}, in this paper we are able to remove the restrictive
algebraicity assumption. The following theorem is our main result
for Bernoulli convolutions.
\begin{thm}
\label{thm:main bernoulli thm}Let $d\in\mathbb{Z}_{>0}$ and $(\lambda_{1},...,\lambda_{d})=\lambda\in(0,1)^{d}$
be with $\lambda_{1}>...>\lambda_{d}$, and suppose that $P(\lambda_{j})\ne0$
for every $1\le j\le d$ and nonzero polynomial $P$ with coefficients
$\pm1,0$. Then $\dim\mu_{\lambda}=\min\left\{ \dim_{L}\mu_{\lambda},d\right\} $.
\end{thm}

In the next subsection, we present our results in the more general
setup of homogeneous diagonal self-affine systems with rational translations.

\subsection{Main result in the general setup}

Fix $d\ge1$ and let $\Phi=\{\varphi_{i}(x)=A_{i}x+a_{i}\}_{i\in\Lambda}$
be a finite collection of invertible affine contractions of $\mathbb{R}^{d}$.
Such a collection is called an affine iterated function system (IFS).
It is well known (see \cite{Hut}) that there exists a unique nonempty
compact $K_{\Phi}\subset\mathbb{R}^{d}$ with $K_{\Phi}=\cup_{i\in\Lambda}\varphi_{i}(K_{\Phi})$.
It is called the attractor or self-affine set corresponding to $\Phi$.

Certain natural measures are supported on $K_{\Phi}$. Fixing a probability
vector $p=(p_{i})_{i\in\Lambda}$, there exists a unique Borel probability
measure $\mu$ on $\mathbb{R}^{d}$ which satisfies the relation $\mu=\sum_{i\in\Lambda}p_{i}\cdot\varphi_{i}\mu$
(again, see \cite{Hut}), where $\varphi_{i}\mu$ is the push-forward
of $\mu$ via $\varphi_{i}$. It is supported on $K_{\Phi}$, and
is called the self-affine measure corresponding to $\Phi$ and $p$.
It is known that self-affine measures are always exact dimensional
(see \cite{FH-dimension,BK,MR4557759}).

Many mathematical problems surround self-affine sets and measures,
but perhaps the most natural one is to determine their dimension.
It has been studied by many authors, and its computation is one of
the major open problems in fractal geometry. In what follows, we denote
the Hausdorff dimension of $E\subset\mathbb{R}^{d}$ by $\dim_{H}E$.

In the 1980s, Falconer \cite{falconer1988hausdorff} introduced a
natural upper bound for the dimension of $K_{\Phi}$, which is called
the affinity dimension. It is denoted by $\dim_{A}\Phi$ and depends
only on the linear parts $\{A_{i}\}_{i\in\Lambda}$. Falconer has
shown that when $\Vert A_{i}\Vert_{\mathrm{op}}<1/2$ for $i\in\Lambda$,
and under a natural randomisation of the translations $\{a_{i}\}_{i\in\Lambda}$,
the equality
\begin{equation}
\dim_{H}K_{\Phi}=\min\left\{ d,\dim_{A}\Phi\right\} \label{eq:dim K =00003D dim_A}
\end{equation}
holds almost surely\footnote{In fact, Falconer proved this with $1/3$ as the upper bound on the
norms; it was subsequently shown by Solomyak \cite{So} that $1/2$
suffices.}.

Similarly, there exists a natural upper bound for the dimension of
$\mu$. It is denoted by $\dim_{L}(\Phi,p)$, is called the Lyapunov
dimension corresponding to $\Phi$ and $p$, and depends only on the
entropy of $p$ and the Lyapunov exponents corresponding to $\{A_{i}\}_{i\in\Lambda}$
and $p$. It has been shown by Jordan, Pollicott and Simon \cite{JPS}
that, under the same conditions as in Falconer's result, the equality
\begin{equation}
\dim\mu=\min\left\{ d,\dim_{L}(\Phi,p)\right\} \label{eq:dim=00005Cmu=00003D dim_L}
\end{equation}
holds almost surely. The definition of the Lyapunov dimension is given
in Section \ref{subsec:setup-and add notat}.

The last results show that (\ref{eq:dim K =00003D dim_A}) and (\ref{eq:dim=00005Cmu=00003D dim_L})
hold typically, but do not provide any explicit examples. It is of
course desirable to find explicit and verifiable conditions under
which these equalities hold. In recent years, and while assuming the
collection $\{A_{i}\}_{i\in\Lambda}$ is strongly irreducible, such
conditions have been obtained when $d=2$ (see \cite{BHR,HR,MoSh})
and $d=3$ (see \cite{MoSe,Rap_SA_Rd}).

An important subclass of self-affine systems, which is in a sense
opposite to the strongly irreducible case, is the one in which the
linear parts $\{A_{i}\}_{i\in\Lambda}$ are all diagonal. In this
paper, we always assume that we are in the diagonal situation. We
further assume that $\Phi$ is homogeneous, which means that $A_{i_{1}}=A_{i_{2}}$
for $i_{1},i_{2}\in\Lambda$. Thus, suppose that there exist $1>\lambda_{1}>...>\lambda_{d}>0$
such that for each $i\in\Lambda$
\begin{equation}
\varphi_{i}(x)=\left(\lambda_{1}x_{1}+a_{i,1},...,\lambda_{d}x_{d}+a_{i,d}\right)\text{ for }(x_{1},...,x_{d})=x\in\mathbb{R}^{d},\label{eq:form of varph_i}
\end{equation}
where $a_{i,1},...,a_{i,d}\in\mathbb{R}$. Note that $\Phi_{j}:=\left\{ t\rightarrow\lambda_{j}t+a_{i,j}\right\} _{i\in\Lambda}$
is a homogeneous affine IFS on $\mathbb{R}$ for each $1\le j\le d$.
\begin{defn*}
A homogeneous affine IFS $\Psi:=\{\psi_{i}\}_{i\in\Lambda}$ on $\mathbb{R}$
is said to be exponentially separated if there exists $c>0$ such
that $\left|\psi_{u_{1}}(0)-\psi_{u_{2}}(0)\right|\ge c^{n}$ for
all $n\ge1$ and distinct $u_{1},u_{2}\in\Lambda^{n}$, where $\psi_{u}:=\psi_{i_{1}}\circ...\circ\psi_{i_{n}}$
for $i_{1}...i_{n}=u\in\Lambda^{n}$.
\end{defn*}
In \cite{Rap_SA_diag}, the first author has shown that (\ref{eq:dim K =00003D dim_A})
and (\ref{eq:dim=00005Cmu=00003D dim_L}) hold whenever $\Phi_{1},...,\Phi_{d}$
are exponentially separated. In this paper, we show that when the
numbers $a_{i,j}$ are rational, the exponential separation assumption
can be relaxed considerably.
\begin{defn*}
We say that an affine IFS $\Psi:=\{\psi_{i}\}_{i\in\Lambda}$ on $\mathbb{R}$
has no exact overlaps if its elements generate a free semigroup. That
is, if $\psi_{u_{1}}\ne\psi_{u_{2}}$ for all distinct $u_{1},u_{2}\in\Lambda^{*}$,
where $\Lambda^{*}$ is the set of finite words over $\Lambda$.
\end{defn*}
\begin{rem*}
A homogeneous IFS $\Psi$ clearly has no exact overlaps whenever it
is exponentially separated. When $\Psi$ is defined by algebraic parameters,
the converse also holds true. In general, as shown by examples constructed
in \cite{baker2019iterated,BaranyKaenmakiSuper}, exponential separation
is a more restrictive condition than the absence of exact overlaps.
\end{rem*}
We can now state our main result, which is the following statement.
\begin{thm}
\label{thm:main gen thm}Let $\Phi=\{\varphi_{i}\}_{i\in\Lambda}$
be a homogeneous diagonal affine IFS on $\mathbb{R}^{d}$ with rational
translations. More precisely, suppose that there exist $1>\lambda_{1}>...>\lambda_{d}>0$
and $\left\{ a_{i,j}\::\:i\in\Lambda\text{ and }1\le j\le d\right\} \subset\mathbb{Q}$
such that $\varphi_{i}$ is of the form (\ref{eq:form of varph_i})
for each $i\in\Lambda$. Let $p=(p_{i})_{i\in\Lambda}$ be a probability
vector and let $\mu$ be the self-affine measure associated to $\Phi$
and $p$. Then, assuming $\Phi_{j}:=\left\{ t\rightarrow\lambda_{j}t+a_{i,j}\right\} _{i\in\Lambda}$
has no exact overlaps for each $1\le j\le d$, we have $\dim_{H}K_{\Phi}=\min\left\{ d,\dim_{A}\Phi\right\} $
and $\dim\mu=\min\left\{ d,\dim_{L}(\Phi,p)\right\} $.
\end{thm}

\begin{rem*}
When $d=1$, Theorem \ref{thm:main gen thm} has been established
in \cite[Appendix A]{rapaport20203maps} by directly extending \cite{Var-Bernoulli}.
In \cite{FengFeng}, it has been recently shown that for $d=1$ the
theorem also holds with algebraic translations. It is expected that
this extension should be possible for general $d\ge1$, but we do
not pursue it here.
\end{rem*}
\begin{rem*}
Note that Theorem \ref{thm:main bernoulli thm} follows directly from
Theorem \ref{thm:main gen thm} by considering the IFS
\[
\left\{ (x_{1},...,x_{d})\rightarrow\left(\lambda_{1}x_{1},...,\lambda_{d}x_{d}\right)\pm(1,...,1)\right\} 
\]
and probability vector $(1/2,1/2)$.
\end{rem*}
As demonstrated by various carpet-like examples (see e.g. \cite{MR2298824,Bed,MR2947936,MR1183358,Mcm}),
it is necessary to assume that the systems $\Phi_{j}$ have no exact
overlaps. On the other hand, it is reasonable to expect Theorem \ref{thm:main gen thm}
to remain true without assuming the translations $a_{i,j}$ are rational
(or algebraic). Unfortunately, at this point, this is well beyond
our reach. Indeed, this has not been achieved even when $d=1$, in
which case the validity of such a statement is considered one of the
major open problems in fractal geometry (see \cite{MR3966837,Var_ICM}).

\subsection{\label{subsec:About-the-proof}About the proof}

In this subsection, we present the general outline of the proof of
Theorem \ref{thm:main gen thm}. For the rest of the paper, fix a
finite nonempty index set $\Lambda$ and an integer $d\ge1$. The
proof is carried out by induction on $d$. Thus, assume that the theorem
holds whenever the dimension of the ambient space is strictly less
than $d$.

By rescaling the IFS if necessary, in order to prove Theorem \ref{thm:main gen thm},
it suffices to consider the case in which the translations are integers.
Thus, from now on, fix
\[
\left\{ a_{i,j}\::\:i\in\Lambda\text{ and }1\le j\le d\right\} \subset\mathbb{Z}.
\]
We also fix a probability vector $p=(p_{i})_{i\in\Lambda}$.

Write
\[
\Omega:=\left\{ (\eta_{1},...,\eta_{d})\in(0,1)^{d}\::\:\eta_{1}>...>\eta_{d}\right\} ,
\]
and for each $(\eta_{1},...,\eta_{d})=\eta\in\Omega$ let $\Phi^{\eta}$
be the homogeneous diagonal affine IFS on $\mathbb{R}^{d}$ with translations
$a_{i}:=(a_{i,1},...,a_{i,d})$ and linear part $\mathrm{diag}(\eta_{1},...,\eta_{d})$.
That is,
\[
\Phi^{\eta}:=\left\{ \varphi_{i}^{\eta}(x)=\left(\eta_{1}x_{1}+a_{i,1},...,\eta_{d}x_{d}+a_{i,d}\right)\right\} _{i\in\Lambda}.
\]
Fix $(\lambda_{1},...,\lambda_{d})=\lambda\in\Omega$ such that $\Phi_{j}^{\lambda}:=\left\{ t\rightarrow\lambda_{j}t+a_{i,j}\right\} _{i\in\Lambda}$
has no exact overlaps for each $1\le j\le d$, and let $\mu_{\lambda}$
be the self-affine measure associated to $\Phi^{\lambda}$ and $p$.

By considering $p$ with equal weights, (\ref{eq:dim K =00003D dim_A})
follows from (\ref{eq:dim=00005Cmu=00003D dim_L}). Thus, in order
to prove the theorem, we aim to show that $\dim\mu_{\lambda}=\min\left\{ d,\dim_{L}(\Phi^{\lambda},p)\right\} $.
Assume by contradiction that this is not the case. Hence, by \cite[Theorem 1.7]{Rap_SA_diag},
it must hold that $\lambda_{j_{0}}$ is transcendental for some $1\le j_{0}\le d$.

Write $[d]=\{1,...,d\}$, and for each $\emptyset\ne J\subset[d]$
denote by $\pi_{J}$ the orthogonal projection onto $\mathrm{span}\{e_{j}\}_{j\in J}$,
where $\{e_{j}\}_{j\in[d]}$ is the standard basis of $\mathbb{R}^{d}$.
Note that $\pi_{J}\mu_{\lambda}:=\mu_{\lambda}\circ\pi_{J}^{-1}$
is (an embedded copy of) a self-affine measure associated to $p$
and an affine IFS on $\mathbb{R}^{|J|}$ for which the conditions
of Theorem \ref{thm:main gen thm} are also satisfied. From this observation,
by the induction assumption, and by applying the Ledrappier--Young
formula for diagonal self-affine measures (obtained in \cite{FH-dimension}),
it can be shown that $\dim\mu_{\lambda}=\dim_{L}(\Phi^{\lambda},p)$
whenever $\dim\pi_{J}\mu_{\lambda}<|J|$ for some proper subset $J$
of $[d]$. Thus, it must hold that $\dim\pi_{J}\mu_{\lambda}=|J|$
for each $\emptyset\ne J\subsetneq[d]$.

From this point, we try to follow the argument of Varjú \cite{Var-Bernoulli},
which roughly speaking relies on three main components. The first
component is the work of Hochman \cite{Ho1} regarding exponentially
separated systems on the real line. Extending this work to higher
dimensions requires significant new ideas, but this was already achieved
in \cite{Rap_SA_diag}. By applying results from \cite{Rap_SA_diag},
we are able to show that, under our assumptions of dimension drop
and full dimensionality of projections, the parameters $\lambda_{1},...,\lambda_{d}$
can be approximated by algebraic numbers, of controllable degree and
height, with high precision at all scales (see Proposition \ref{prop:exists polynomials}).

The second component is the connection, obtained by Breuillard and
Varjú \cite{BV-entropy}, between random walk entropy and Mahler measure.
For an affine IFS $\Psi=\{\psi_{i}\}_{i\in\Lambda}$, we write $h_{RW}(\Psi,p)$
for the entropy of the random walk generated by $\Psi$ and $p$ (see
Section \ref{subsec:Random-walk-entropy}). The definition of the
Mahler measure of an algebraic number is provided in Section \ref{subsec:Mahler-measure-and}.
For $(\eta_{1},...,\eta_{d})=\eta\in\Omega$ such that $\eta_{j_{0}}$
is algebraic, it follows from the results of \cite{BV-entropy} that
$h_{RW}(\Phi_{j_{0}}^{\eta},p)$ is close to its maximal possible
value $H(p)$ whenever the Mahler measure of $\eta_{j_{0}}$ is sufficiently
large (see Theorem \ref{thm:lb on RW-ent for large Mahler}). Here,
$\Phi_{j_{0}}^{\eta}:=\left\{ t\rightarrow\eta_{j_{0}}t+a_{i,j_{0}}\right\} _{i\in\Lambda}$.
Since $h_{RW}(\Phi^{\eta},p)$ is always at least as large as $h_{RW}(\Phi_{j_{0}}^{\eta},p)$,
we are able to use this result without any adaptation needed.

The third component is the work \cite{BV-transcendent} of Breuillard
and Varjú regarding Bernoulli convolutions on $\mathbb{R}$, in which
they establish that parameters with dimension drop can be approximated
extremely well, at infinitely many scales, by controllable algebraic
parameters that also have dimension drop. Most of the present paper
is dedicated to extending this result to higher dimensions, which
is the content of the following statement. Set
\[
L_{0}:=\max\left\{ a_{i_{1},j}-a_{i_{2},j}\::\:i_{1},i_{2}\in\Lambda\text{ and }1\le j\le d\right\} ,
\]
and for $n\ge1$ let $\mathcal{P}_{L_{0}}^{(n)}\subset\mathbb{Z}[X]$
be the set of polynomials of degree strictly less than $n$ with integer
coefficients bounded in absolute value by $L_{0}$. For $1\le j\le d$
write $\chi_{j}:=-\log\lambda_{j}$, and set
\[
\kappa:=\chi_{d}\dim\mu_{\lambda}-\sum_{j=1}^{d-1}(\chi_{d}-\chi_{j}).
\]

\begin{thm}
\label{thm:alg approx}Suppose that $\dim\mu_{\lambda}<\min\left\{ d,\dim_{L}(\Phi^{\lambda},p)\right\} $,
$\dim\pi_{J}\mu_{\lambda}=|J|$ for each proper subset $J$ of $[d]$,
and $\lambda_{j_{0}}$ is transcendental for some $1\le j_{0}\le d$.
Then for every $\epsilon>0$ and $N\ge1$ there exist $n\ge N$ and
$(\eta_{1},\ldots,\eta_{d})=\eta\in\Omega$ such that,
\begin{enumerate}
\item for each $1\le j\le d$ there exists $0\neq P_{j}\in\mathcal{P}_{L_{0}}^{(n)}$
with $P_{j}(\eta_{j})=0$;
\item \label{enu:ineq for H_RW}$h_{RW}(\Phi^{\eta},p)<\kappa+\epsilon$;
\item $|\lambda-\eta|\le\exp\big(-n^{1/\epsilon}\big)$.
\end{enumerate}
\end{thm}

\begin{rem*}
The assumption regarding $\lambda_{j_{0}}$ being transcendental for
some $1\le j_{0}\le d$ is not really necessary. It simplifies the
proof slightly by enabling us to ignore the case of algebraic parameters,
which is treated in \cite{Rap_SA_diag}.
\end{rem*}
\begin{rem*}
In \cite{BV-transcendent}, the inequality $h_{RW}(\Phi^{\eta},p)<\kappa+\epsilon$
is replaced with $\dim\mu_{\eta}<\dim\mu_{\lambda}+\epsilon$, where
$\mu_{\eta}$ is the Bernoulli convolution associated to $\eta$.
When $d=1$ and $\dim\mu_{\lambda}+\epsilon/\chi_{1}<1$, it follows
by the work of Hochman that these inequalities are equivalent (see
\cite[Section 3.4]{BV-entropy}). The intuition behind the value $\kappa$
is best explained by \cite[Lemma 4.1]{Rap_SA_diag}, in which it is
shown that $\kappa$ equals the limit of the normalized entropies
$\frac{1}{n}H\left(\mu_{\lambda},\mathcal{E}_{n}\right)$, where $\mathcal{E}_{n}$
is the level-$n$ non-conformal partition of $\mathbb{R}^{d}$ determined
by $\lambda$ (see Section \ref{subsec:Non-conformal-partitions}).
\end{rem*}
Once we have the aforementioned three components at our disposal,
we can complete the proof of Theorem \ref{thm:main gen thm} by following
the proof of \cite[Theorem A.1]{rapaport20203maps}, which handles
the case $d=1$ and is a modification of the argument found in \cite{Var-Bernoulli}.

Let us next discuss the proof of Theorem \ref{thm:alg approx}. The
key ingredient of the proof is an entropy increase statement, in which
the amount of increase in entropy is explicit. For the case $d=1$,
such a statement was obtained by Varjú \cite[Theorem 3]{varju2019absolute},
and it plays a key role in the proof of the main result of \cite{BV-transcendent}.

In order to state our entropy increase result, we need some preparations.
Given a discrete random vector $Y$, its Shannon entropy is denoted
by $H(Y)$. For a bounded random vector $X=(X_{1},...,X_{d})$ in
$\mathbb{R}^{d}$ and $(r_{1},...,r_{d})=r\in\mathbb{R}_{>0}^{d}$,
we set
\begin{equation}
H\left(X;r\right):=\int_{[0,1)^{d}}H\left(\left\lfloor X_{1}/r_{1}+x_{1}\right\rfloor ,...,\left\lfloor X_{d}/r_{d}+x_{d}\right\rfloor \right)\:dx_{1}...dx_{d}.\label{eq:def of avg ent in intro}
\end{equation}
We refer to $H\left(X;r\right)$ as the average entropy of $X$ at
scale $r$. Given $r'\in\mathbb{R}_{>0}^{d}$, we also set
\[
H\left(X;r\mid r'\right):=H\left(X;r\right)-H\left(X;r'\right).
\]
If $\mu$ is the distribution of $X$, we write $H\left(\mu;r\right)$
and $H\left(\mu;r\mid r'\right)$ in place of $H\left(X;r\right)$
and $H\left(X;r\mid r'\right)$. These quantities originate in the
work of Wang \cite{MR2824843}, and in the case $d=1$ they play an
important role in the papers \cite{varju2019absolute,BV-transcendent}.
The disappearance of certain error terms, which are present without
the averaging in (\ref{eq:def of avg ent in intro}), is the main
advantage of using the notion of average entropy.

In order to state the entropy increase result, we also need the following
definition. As noted above, for $n\ge0$ we denote by $\mathcal{E}_{n}$
the level-$n$ non-conformal partition determined by $\lambda$. That
is, $\mathcal{E}_{n}$ is a partition of $\mathbb{R}^{d}$ into rectangles
with side lengths roughly $\lambda_{1}^{n},...,\lambda_{d}^{n}$.
We write $\mathcal{M}(\mathbb{R}^{d})$ for the collection of compactly
supported Borel probability measures on $\mathbb{R}^{d}$. Given $\mu\in\mathcal{M}(\mathbb{R}^{d})$
and Borel partitions $\mathcal{C},\mathcal{D}$ of $\mathbb{R}^{d}$,
the conditional entropy of $\mathcal{C}$ given $\mathcal{D}$ with
respect to $\mu$ is denoted by $H\left(\mu,\mathcal{C}\mid\mathcal{D}\right)$
(see Section \ref{subsec:Entropy-of-a partition}).
\begin{defn}
\label{def:non-saturated-measure}Given $\epsilon>0$ and $m\ge1$,
we say that $\mu\in\mathcal{M}(\mathbb{R}^{d})$ is $(\epsilon,m)$-non-saturated
across the principal directions at all scales, or simply $(\epsilon,m)$-non-saturated,
if for all $1\le j\le d$ and $n\ge0$,
\[
\frac{1}{m}H\left(\mu,\mathcal{E}_{n+m}\mid\mathcal{E}_{n}\vee\pi_{[d]\setminus\{j\}}^{-1}\mathcal{E}_{n+m}\right)<\chi_{j}-\epsilon.
\]
\end{defn}

\begin{rem*}
Given $\mu\in\mathcal{M}(\mathbb{R}^{d})$, $n\ge0$ and $E\in\mathcal{E}_{n}$,
we refer to $\mu_{E}:=\frac{1}{\mu(E)}\mu|_{E}$ as a (non-conformal)
$n$-component of $\mu$. It follows from basic properties of entropy
that for $m\ge1$ and $1\le j\le d$,
\[
\frac{1}{m}H\left(\mu_{E},\mathcal{E}_{n+m}\mid\pi_{[d]\setminus\{j\}}^{-1}\mathcal{E}_{n+m}\right)\le\chi_{j}+O\left(\frac{1}{m}\right).
\]
Thus, roughly speaking, the assumption of $\mu$ being $(\epsilon,m)$-non-saturated
means that for all $1\le j\le d$ and all scales $n\ge0$, for a non-negligible
proportion of $n$-components $\mu_{E}$, the entropy along narrow
rectangular tubes in direction $e_{j}\mathbb{R}$ is not full.
\end{rem*}
We can now state our entropy increase result. For $t\in\mathbb{R}$,
we write $\lambda^{t}:=\left(\lambda_{1}^{t},...,\lambda_{d}^{t}\right)$.
\begin{thm}
\label{thm:effective ent inc result}For each $\epsilon>0$ and $M\ge1$,
there exists $C=C(\lambda,\epsilon,M)>1$ such that the following
holds. Let $\mu\in\mathcal{M}(\mathbb{R}^{d})$ be $(\epsilon,m)$-non-saturated
for all $m\ge M$, and let $\nu\in\mathcal{M}(\mathbb{R}^{d})$, $0<\beta<1/2$,
and $t_{2}>t_{1}>0$ be with $\frac{1}{t_{2}-t_{1}}H\left(\nu;\lambda^{t_{2}}\mid\lambda^{t_{1}}\right)>\beta$.
Then,
\[
H\left(\nu*\mu;\lambda^{t_{2}}\mid\lambda^{t_{1}}\right)\ge H\left(\mu;\lambda^{t_{2}}\mid\lambda^{t_{1}}\right)+C^{-1}\beta\left(\log\beta^{-1}\right)^{-1}(t_{2}-t_{1})-C.
\]
\end{thm}

\begin{rem*}
It is easy to see that by replacing $H\left(\cdot;\lambda^{t_{2}}\mid\lambda^{t_{1}}\right)$
with $H\left(\cdot,\mathcal{E}_{\left\lfloor t_{2}\right\rfloor }\mid\mathcal{E}_{\left\lfloor t_{1}\right\rfloor }\right)$
in the last theorem, one gets a formally equivalent statement. On
the other hand, the use of average entropy is crucial to the proof
of Theorem \ref{thm:effective ent inc result}, and also to the proof
of Theorem \ref{thm:alg approx} in which Theorem \ref{thm:effective ent inc result}
is applied repeatedly. This is due to the aforementioned disappearance
of error terms caused by the averaging procedure.
\end{rem*}
\begin{rem*}
Since the IFS $\Phi^{\lambda}$ is homogeneous, the measures $\mu_{\lambda}$
can be represented as an infinite convolution (see Section \ref{subsec:Convolution-factors}).
Assuming $\dim\mu_{\lambda}<d$ and $\dim\pi_{J}\mu_{\lambda}=|J|$
for each proper subset $J$ of $[d]$, we shall show in Section \ref{subsec:Non-saturation-of conv factors}
that there exist $\epsilon>0$ and $M\ge1$ such that $\mu_{\lambda}$
and its convolution factors are all $(\epsilon,m)$-non-saturated
for all $m\ge M$. This enables the repeated application of Theorem
\ref{thm:effective ent inc result} in the proof of Theorem \ref{thm:alg approx}.
\end{rem*}
Given that we have Theorem \ref{thm:effective ent inc result} at
our disposal, the proof of Theorem \ref{thm:alg approx} is similar
to the proof of \cite[Theorem A.2]{rapaport20203maps}, which is a
modification of the argument in \cite{BV-transcendent}. On the other
hand, the proof of Theorem \ref{thm:effective ent inc result}, which
is an extension of the proof of \cite[Theorem 3]{varju2019absolute}
to higher dimensions, requires new ideas and is probably the main
technical contribution of this paper.

We say that $\zeta\in\mathcal{M}(\mathbb{R}^{d})$ is a Bernoulli
measure if it is of the form $\zeta=\frac{1}{2}(\delta_{x}+\delta_{y})$
for some distinct $x,y\in\mathbb{R}^{d}$. The proof of \cite[Theorem 3]{varju2019absolute},
which is motivated by the work of Bourgain \cite{MR2763000}, consists
of three steps. In the first step, given a scale $t>0$ and $\nu\in\mathcal{M}(\mathbb{R})$,
a decomposition of $\nu$ is obtained as a convex combination of probability
measures in which a non-negligible part of the mass, in a manner depending
on $t$ and $\nu$, is captured by Bernoulli measures of diameter
roughly $t$. In the second step, entropy increase is obtained for
the convolution of a non-saturated measure with a Bernoulli measure.
More precisely, given $t>0$ and $\mu\in\mathcal{M}(\mathbb{R})$,
it is shown that
\begin{equation}
H\left(\zeta*\mu;r_{2}\mid r_{1}\right)>H\left(\mu;r_{2}\mid r_{1}\right)+\delta,\label{eq:ent increse bernoulli for d=00003D1}
\end{equation}
where $\zeta\in\mathcal{M}(\mathbb{R})$ is Bernoulli, $\mathrm{diam}(\mathrm{supp}(\zeta))$
and $0<r_{2}<r_{1}$ are comparable with $t$, and the increase $\delta>0$
depends quantitatively on the non-saturation of $\mu$ at scale $t$.
Finally, in the third step, the first two steps are applied repeatedly
while averaging over the scale, thus obtaining the desired entropy
increase.

Extending the third step to higher dimensions requires minimal changes.
The extension of the first step requires some work, but the core idea
remains the same. On the other hand, the argument given in \cite{varju2019absolute}
for the second step fails completely in higher dimensions, and a new
approach is needed.

Our approach for extending the second step involves using ideas from
\cite{Rap_SA_diag}, which rely on the Berry-Esseen theorem and extend
the proof of Hochman's inverse theorem from \cite{Ho1} to the higher
dimensional non-conformal setup. These ideas will enable us to obtain
entropy increase for convolutions with repeated self-convolutions
of a Bernoulli measure. In order to deduce from this entropy increase
for convolutions with a Bernoulli measure, we establish a version
for average entropy of a classical lemma due to Kaimanovich and Vershik
\cite{kaimanovich_and_vershik} (see Lemma \ref{lem:KV by submodularity}
below). This version, which crucially has no error term, follows from
a connection between average and differential entropies (see Lemma
\ref{lem:avg ent as dif ent}), and from an entropy submodularity
inequality due to Madiman \cite{madiman2008entropy}. The inequality
also plays an important role in the earlier works \cite{varju2019absolute,BV-entropy,BV-transcendent},
but the average entropy version of the Kaimanovich--Vershik lemma
seems to be new.

\subsubsection*{\textbf{\emph{Structure of the paper}}}

The rest of the paper is organised as follows. In Section \ref{sec:Preliminaries}
we introduce some notations, define necessary concepts, and develop
basic properties of average entropy in $\mathbb{R}^{d}$. Section
\ref{sec:An-entropy-increase result} is devoted to the proof of the
entropy increase result, Theorem \ref{thm:effective ent inc result}.
In Section \ref{sec:Algebraic-approximation} we establish the algebraic
approximation statement, Theorem \ref{thm:alg approx}. We prove our
main result, Theorem \ref{thm:main gen thm}, in Section \ref{sec:Proof-of-main result}.

\subsubsection*{\textbf{\emph{Acknowledgment}}}

We would like to thank Péter Varjú for his comments on an early version
of this paper.

\section{\label{sec:Preliminaries}Preliminaries}

\subsection{\label{subsec:Basic-notations}Basic notations}

Throughout the paper, the base of the $\log$ and $\exp$ functions
is always $2$. For an integer $k\ge0$ we write $[k]$ to represent
the set $\{1,...,k\}$, with the convention that $[0]=\emptyset$.
Given a metric space $X$, denote by $\mathcal{M}(X)$ the collection
of compactly supported Borel probability measures on $X$. For $n,l\in\mathbb{Z}_{>0}$,
let $\mathcal{P}_{l}^{(n)}\subset\mathbb{Z}[X]$ be the set of polynomials
of degree strictly less than $n$ with integer coefficients bounded
in absolute value by $l$.

Given $R_{1},R_{2}\in\mathbb{R}$ with $R_{1},R_{2}\ge1$, we write
$R_{1}\ll R_{2}$ in order to indicate that $R_{2}$ is large with
respect to $R_{1}$. Formally, this means that $R_{2}\ge f(R_{1})$,
where $f$ is an unspecified function from $[1,\infty)$ into itself.
The values attained by $f$ are assumed to be sufficiently large in
a manner depending on the specific context.

Similarly, given $0<\epsilon_{1},\epsilon_{2}<1$ we write $R_{1}\ll\epsilon_{1}^{-1}$,
$\epsilon_{2}^{-1}\ll R_{2}$, and $\epsilon_{1}^{-1}\ll\epsilon_{2}^{-1}$
to respectively indicate that $\epsilon_{1}$ is small with respect
to $R_{1}$, $R_{2}$ is large with respect to $\epsilon_{2}$, and
$\epsilon_{2}$ is small with respect to $\epsilon_{1}$.

The relation $\ll$ is clearly transitive. That is, if $R_{1}\ll R_{2}$
and for $R_{3}\ge1$ we have $R_{2}\ll R_{3}$, then also $R_{1}\ll R_{3}$.
For instance, the sentence `Let $m\ge1$, $k\ge K(m)\ge1$ and $n\ge N(m,k)\ge1$
be given' is equivalent to `Let $m,k,n\ge1$ be with $m\ll k\ll n$'.

\subsection{\label{subsec:setup-and add notat}Setup and related notations}

Recall from Section \ref{subsec:About-the-proof} that throughout
the paper we fix $d\in\mathbb{Z}_{>0}$, a finite nonempty index set
$\Lambda$, a probability vector $p=(p_{i})_{i\in\Lambda}$, and integers
\[
\left\{ a_{i,j}\::\:i\in\Lambda\text{ and }1\le j\le d\right\} .
\]
Set
\[
a_{i}:=(a_{i,1},...,a_{i,d})\text{ for }i\in\Lambda,
\]
and recall that
\begin{equation}
L_{0}:=\max\left\{ a_{i_{1},j}-a_{i_{2},j}\::\:i_{1},i_{2}\in\Lambda\text{ and }1\le j\le d\right\} .\label{eq:def of L_0}
\end{equation}

Write
\[
\Omega:=\left\{ (\eta_{1},...,\eta_{d})\in(0,1)^{d}\::\:\eta_{1}>...>\eta_{d}\right\} ,
\]
and for $(\eta_{1},...,\eta_{d})=\eta\in\Omega$, $i\in\Lambda$ and
$1\le j\le d$, let $\varphi_{i,j}^{\eta}:\mathbb{R}\rightarrow\mathbb{R}$
be defined by $\varphi_{i,j}^{\eta}(t)=\eta_{j}t+a_{i,j}$. Given
$\emptyset\ne J\subset[d]$ denote by $\Phi_{J}^{\eta}$ the IFS $\{\varphi_{i,J}^{\eta}\}_{i\in\Lambda}$
on $\mathbb{R}^{J}$, where 
\[
\varphi_{i,J}^{\eta}(x)=\left(\varphi_{i,j}^{\eta}(x_{j})\right)_{j\in J}\text{ for }i\in\Lambda\text{ and }(x_{j})_{j\in J}=x\in\mathbb{R}^{J}.
\]
For $1\le j\le d$ we write $\Phi_{j}^{\eta}$ in place of $\Phi_{\{j\}}^{\eta}$.
We also write $\Phi^{\eta}$ in place of $\Phi_{[d]}^{\eta}$, and
$\varphi_{i}^{\eta}$ in place of $\varphi_{i,[d]}^{\eta}$ for $i\in\Lambda$.

Let $\mu_{\eta}\in\mathcal{M}(\mathbb{R}^{d})$ denote the self-affine
measure generated by $\Phi^{\eta}$ and $p$. For $n\in\mathbb{Z}_{>0}$
set
\[
\mu_{\eta}^{(n)}:=\sum_{u\in\Lambda^{n}}p_{u}\delta_{\varphi_{u}^{\eta}(0)}\in\mathcal{M}(\mathbb{R}^{d}),
\]
where $p_{u}:=p_{i_{1}}\cdot...\cdot p_{i_{n}}$ and $\varphi_{u}^{\eta}:=\varphi_{i_{1}}^{\eta}\circ...\circ\varphi_{i_{n}}^{\eta}$
for $i_{1}...i_{n}=u\in\Lambda^{n}$. The measure $\mu_{\eta}^{(n)}$
may be thought of as the discrete level-$n$ approximation of $\mu_{\eta}$.

From now on, we fix $(\lambda_{1},...,\lambda_{d})=\lambda\in\Omega$
for the rest of the paper. Since $\lambda$ is fixed, we sometimes
do not explicitly indicate the dependence on $\lambda$ of various
parameters. Similarly, we also do not explicitly indicate dependence
on $\{a_{i,j}\}$ and $p$.

For $1\le j\le d$ set $\chi_{j}:=-\log\lambda_{j}$, and let $H(p):=-\sum_{i\in\Lambda}p_{i}\log p_{i}$
denote the entropy of $p$. Setting
\[
m=m(\lambda,p):=\max\left\{ 0\le j\le d\::\:\chi_{1}+...+\chi_{j}\le H(p)\right\} ,
\]
the Lyapunov dimension corresponding to $\Phi^{\lambda}$ and $p$
is defined by
\[
\dim_{L}(\Phi^{\lambda},p):=\begin{cases}
m+\frac{H(p)-\chi_{1}-...-\chi_{m}}{\chi_{m+1}} & ,\text{ if }m<d\\
d\frac{H(p)}{\chi_{1}+...+\chi_{d}} & ,\text{ if }m=d
\end{cases}.
\]

For brevity, we use the notation
\[
\gamma:=\min\left\{ d,\dim_{L}(\Phi^{\lambda},p)\right\} \text{ and }\kappa:=\chi_{d}\dim\mu_{\lambda}-\sum_{j=1}^{d-1}(\chi_{d}-\chi_{j}).
\]
It is easy to verify that for each $0\le j<d$,
\begin{equation}
\gamma\le j+\frac{H(p)-\chi_{1}-...-\chi_{j}}{\chi_{j+1}}.\label{eq:ub for LY-dim}
\end{equation}

\subsection{Algebraic notations}

Let $\{e_{1},...,e_{d}\}$ be the standard basis of $\mathbb{R}^{d}$.
Given $J\subset[d]$ denote by $\pi_{J}$ the orthogonal projection
onto $\mathrm{span}\{e_{j}\::\:j\in J\}$. Thus,
\[
\pi_{J}(x)=\sum_{j\in J}\left\langle e_{j},x\right\rangle e_{j}\text{ for }x\in\mathbb{R}^{d},
\]
and $\pi_{[0]}$ is identically $0$. We write $\pi_{j}$ in place
of $\pi_{\{j\}}$ for $1\le j\le d$. Given $\eta\in\Omega$ and $\emptyset\ne J\subset[d]$,
note that $\pi_{J}\mu_{\eta}$ is (an embedded copy of) the self-affine
measure corresponding to $\Phi_{J}^{\eta}$ and $p$, where $\pi_{J}\mu_{\eta}:=\mu_{\eta}\circ\pi_{J}^{-1}$
is the push-forward of $\mu_{\eta}$ via $\pi_{J}$. In particular,
$\pi_{J}\mu_{\eta}$ is exact dimensional.

Let $\mathbb{R}_{>0}$ be the group of positive real numbers. For
$r=(r_{1},...,r_{d})$ and $r'=(r_{1}',...,r_{d}')$ in $\mathbb{R}_{>0}^{d}$,
set
\[
rr':=(r_{1}r_{1}',...,r_{d}r_{d}')\text{ and }r^{-1}:=(r_{1}^{-1},...,r_{d}^{-1}),
\]
which makes $\mathbb{R}_{>0}^{d}$ into a multiplicative group. We
write $r'\le r$ whenever $r'_{j}\le r_{j}$ for each $1\le j\le d$,
thereby defining a partial order on $\mathbb{R}_{>0}^{d}$.

For $(x_{1},...,x_{d})=x\in\mathbb{R}^{d}$ set $rx:=(r_{1}x_{1},...,r_{d}x_{d})$,
which defines an action of $\mathbb{R}_{>0}^{d}$ on $\mathbb{R}^{d}$.
We sometimes write $xr$ in place of $rx$. Let $S_{r}:\mathbb{R}^{d}\rightarrow\mathbb{R}^{d}$
be defined by $S_{r}(x):=rx$. We often write $\det r$ instead of
$\det S_{r}$. Additionally, we often write $r'/r$ and $x/r$ in
place of $r^{-1}r'$ and $r^{-1}x$.

Set
\[
\left\lfloor x\right\rfloor :=\left(\left\lfloor x_{1}\right\rfloor ,...,\left\lfloor x_{d}\right\rfloor \right)\text{ and }r^{t}:=(r_{1}^{t},...,r_{d}^{t})\text{ for }t\in\mathbb{R},
\]
where $\left\lfloor t\right\rfloor $ is the integral part of $t\in\mathbb{R}$.
We denote by $|r|$ and $|x|$ the Euclidean norms of $r$ and $x$.
That is, $|r|:=\left(r_{1}^{2}+...+r_{d}^{2}\right)^{1/2}$ and similarly
for $|x|$. For $y\in\mathbb{R}^{d}$ we set $T_{x}(y):=x+y$.

\subsection{Notions of entropy}

\subsubsection{\label{subsec:Entropy-of-a partition}Entropy of a partition}

Let $(S,\mathcal{F})$ be a measurable space. Given a probability
measure $\theta$ on $S$ and a countable partition $\mathcal{D}\subset\mathcal{F}$
of $S$, the entropy of $\theta$ with respect to $\mathcal{D}$ is
defined by
\[
H(\theta,\mathcal{D}):=-\sum_{D\in\mathcal{D}}\theta(D)\log\theta(D).
\]
If $\mathcal{E}\subset\mathcal{F}$ is another countable partition
of $S$, the conditional entropy given $\mathcal{E}$ is defined as
follows:
\[
H(\theta,\mathcal{D}\mid\mathcal{E}):=\sum_{E\in\mathcal{E}}\theta(E)\cdot H(\theta_{E},\mathcal{D}),
\]
where $\theta_{E}:=\theta(E)^{-1}\theta|_{E}$ for $E\in\mathcal{E}$
with $\theta(E)>0$.

For basic properties of entropy and conditional entropy of a partition,
we refer the reader to \cite[Section 3.1]{Ho1}. These basic properties
will often be used without further reference.

\subsubsection{Average entropy}

For a bounded random vector $X$ in $\mathbb{R}^{d}$ and $r\in\mathbb{R}_{>0}^{d}$
we set
\begin{equation}
H\left(X;r\right):=\int_{[0,1)^{d}}H\left(\left\lfloor X/r+x\right\rfloor \right)\:dx,\label{eq:def of avg ent}
\end{equation}
where $H\left(\left\lfloor X/r+x\right\rfloor \right)$ denotes the
Shannon entropy of the discrete random variable $\left\lfloor X/r+x\right\rfloor $.
We call $H\left(X;r\right)$ the average entropy of $X$ at scale
$r$. Given $r'\in\mathbb{R}_{>0}^{d}$, we also set
\[
H\left(X;r\mid r'\right):=H\left(X;r\right)-H\left(X;r'\right).
\]
If $\mu$ is the distribution of $X$, we write $H\left(\mu;r\right)$
and $H\left(\mu;r\mid r'\right)$ in place of $H\left(X;r\right)$
and $H\left(X;r\mid r'\right)$. Basic properties of average entropy
are provided in Section \ref{subsec:Basic-properties-of avg ent}.

\subsubsection{Differential entropy}

Let $F:\mathbb{R}_{\ge0}\rightarrow\mathbb{R}$ be defined by $F(t)=-t\log t$
for $t>0$ and $F(0)=0$. If $X$ is an absolutely continuous random
vector in $\mathbb{R}^{d}$ with density $f:\mathbb{R}^{d}\rightarrow\mathbb{R}_{\ge0}$,
we write $H(X)$ for its differential entropy. That is,
\[
H(X):=\int F(f(x))\:dx.
\]
Note that,
\begin{equation}
H(x+rX)=H(X)+\log\det r\;\text{ for }r\in\mathbb{R}_{>0}^{d}\text{ and }x\in\mathbb{R}^{d}.\label{eq:trans =000026 scaling of diff ent}
\end{equation}

\subsubsection{\label{subsec:Random-walk-entropy}Random walk entropy}

For an affine IFS $\Psi=\{\psi_{i}\}_{i\in\Lambda}$, write $h_{RW}(\Psi,p)$
for the entropy of the random walk generated by $\Psi$ and $p$.
That is,
\[
h_{RW}(\Psi,p):=\lim_{n\to\infty}\frac{1}{n}H\left(\sum_{u\in\Lambda^{n}}p_{u}\delta_{\psi_{u}}\right),
\]
where $H(\cdot)$ denotes Shannon entropy of a discrete measure and
the limit exists by subadditivity. Note that by subadditivity we in
fact have,
\begin{equation}
h_{RW}(\Psi,p)=\underset{n\ge1}{\inf}\frac{1}{n}H\left(\sum_{u\in\Lambda^{n}}p_{u}\delta_{\psi_{u}}\right).\label{eq:h_RW =00003D inf}
\end{equation}

\subsection{\label{subsec:Sequences-with-integral ratios}Sequences with integral
ratios}

In order to avoid certain error terms, which our arguments are unable
to tolerate, it will often be preferable to consider average conditional
entropies of the form $H\left(\mu;r\mid Nr\right)$, where $\mu\in\mathcal{M}\left(\mathbb{R}^{d}\right)$,
$r\in\mathbb{R}_{>0}^{d}$ and $N\in\mathbb{Z}_{>0}^{d}$. For that
reason, in this subsection, we introduce elements $\{s_{n}\}_{n\ge0}\subset\mathbb{R}_{>0}^{d}$
with
\[
s_{n}/s_{n+1}\in\mathbb{Z}_{>0}^{d}\text{ and }\left|s_{n}/\lambda^{n}\right|,\left|\lambda^{n}/s_{n}\right|=O(1)\text{ for all }n\ge0.
\]

For each $1\le j\le d$, let us define by induction a sequence $\{s_{n,j}\}_{n\ge0}\subset(0,1]$
such that $s_{n,j}\ge\lambda_{j}^{n}$ for $n\ge0$. Let $1\le j\le d$
be given, and set $s_{0,j}:=1$. Let $n\ge0$ and suppose that $s_{n,j}$
has already been chosen. Let $b_{n+1,j}$ be the unique positive integer
with 
\[
\frac{s_{n,j}}{b_{n+1,j}}\ge\lambda_{j}^{n+1}>\frac{s_{n,j}}{1+b_{n+1,j}},
\]
and set $s_{n+1,j}=s_{n,j}/b_{n+1,j}$. This completes the inductive
construction of $\{s_{n,j}\}_{n\ge0}$.

For $n\ge0$ we set, 
\begin{equation}
s_{n}:=(s_{n,1},...,s_{n,d}).\label{def:sequ-int}
\end{equation}
By construction, $s_{n}/s_{n'}\in\mathbb{Z}_{>0}^{d}$ for integers
$n'\ge n\ge0$. Moreover, 
\[
\lambda_{j}^{n}\le\frac{s_{n-1,j}}{b_{n,j}}\le2\frac{s_{n-1,j}}{b_{n,j}+1}<2\lambda_{j}^{n},
\]
which gives 
\begin{equation}
\lambda_{j}^{n}\le s_{n,j}<2\lambda_{j}^{n}\text{ for }1\le j\le d\text{ and }n\ge0.\label{eq:prop of sequences}
\end{equation}

\subsection{\label{subsec:Non-conformal-partitions}Non-conformal partitions}

For $n\in\mathbb{Z}$ let $\mathcal{D}_{n}^{\mathbb{R}}$ denote the
level-$n$ dyadic partition of $\mathbb{R}$. That is,
\[
\mathcal{D}_{n}^{\mathbb{R}}:=\left\{ \Bigl[\frac{k}{2^{n}},\frac{k+1}{2^{n}}\Bigr)\::\:k\in\mathbb{Z}\right\} .
\]
Given $t\in\mathbb{R}$, we write $\mathcal{D}_{t}^{\mathbb{R}}$
in place of $\mathcal{D}_{\left\lfloor t\right\rfloor }^{\mathbb{R}}$.

Recall that $\chi_{j}:=-\log\lambda_{j}$ for $1\le j\le d$. For
each $n\in\mathbb{Z}$ define the following non-conformal partition
of $\mathbb{R}^{d}$,
\[
\mathcal{E}_{n}:=\left\{ D_{1}\times...\times D_{d}\::\:D_{j}\in\mathcal{D}_{\chi_{j}n}^{\mathbb{R}}\text{ for }1\le j\le d\right\} .
\]
The partitions $\mathcal{E}_{n}$ play a key role in \cite{Rap_SA_diag}. 

Given $C\ge1$, we say that two Borel partitions $\mathcal{D}$ and
$\mathcal{E}$ of $\mathbb{R}^{d}$ are $C$-commensurable if for
each $D\in\mathcal{D}$ and $E\in\mathcal{E}$,
\[
\#\left\{ E'\in\mathcal{E}\::\:E'\cap D\ne\emptyset\right\} \le C\text{ and }\#\left\{ D'\in\mathcal{D}\::\:D'\cap E\ne\emptyset\right\} \le C.
\]
In this situation,
\[
\left|H\left(\mu,\mathcal{D}\right)-H\left(\mu,\mathcal{E}\right)\right|=O\left(\log C\right)\text{ for \ensuremath{\mu\in\mathcal{M}(\mathbb{R}^{d}).}}
\]

It is easy to verify that for $n\in\mathbb{Z}$, $J\subset[d]$ and
$x\in\mathbb{R}^{d}$,
\begin{equation}
T_{x}^{-1}\pi_{J}^{-1}\mathcal{E}_{n}\text{ and }\pi_{J}^{-1}\mathcal{E}_{n}\text{ are }O(1)\text{-commensurable}.\label{eq:commens part by T_x}
\end{equation}
Additionally, for $k\in\mathbb{Z}$,
\begin{equation}
S_{\lambda^{k}}\pi_{J}^{-1}\mathcal{E}_{n}\text{ and }\pi_{J}^{-1}\mathcal{E}_{n+k}\text{ are }O(1)\text{-commensurable}.\label{eq:commens part by S_=00007Blambda^k=00007D}
\end{equation}
Similarly, by (\ref{eq:prop of sequences}) we get that for $k\in\mathbb{Z}_{\ge0}$,
\begin{equation}
S_{s_{k}}\pi_{J}^{-1}\mathcal{E}_{n}\text{ and }\pi_{J}^{-1}\mathcal{E}_{n+k}\text{ are }O(1)\text{-commensurable}.\label{eq:commens part by s_n}
\end{equation}

\subsection{\label{subsec:Basic-properties-of avg ent}Basic properties of average
entropy}

Let $\mu\in\mathcal{M}(\mathbb{R}^{d})$ be given. It follows from
(\ref{eq:prop of sequences}) that for $k\in\mathbb{Z}_{\ge0}$,
\begin{equation}
H\left(\mu,\mathcal{E}_{k}\right)=H\left(\mu;\lambda^{k}\right)+O(1)=H\left(\mu;s_{k}\right)+O(1).\label{eq:two defs of ent are equiv}
\end{equation}
From (\ref{eq:def of avg ent}) it is immediate that for $r,r'\in\mathbb{R}_{>0}^{d}$,
\begin{equation}
H\left(\mu;r\right)=H\left(S_{r'}\mu;r'r\right).\label{eq:scaling rel of avg ent}
\end{equation}
From (\ref{eq:def of avg ent}) it is also easy to deduce that for
$N\in\mathbb{Z}_{>0}^{d}$
\[
H\left(\mu;N^{-1}r\mid r\right)=\int_{[0,1)^{d}}H\left(\left\lfloor N\left(X/r+x\right)\right\rfloor \mid\left\lfloor X/r+x\right\rfloor \right)\:dx,
\]
where $X$ is a random vector with distribution $\mu$. This gives,
\begin{equation}
0\le H\left(\mu;N^{-1}r\mid r\right)\le\log\left(\det N\right).\label{eq:ub on cond avg ent int ratio}
\end{equation}

The following lemma establishes a connection between average and differential
entropies. Its proof is similar to the proof of \cite[Lemma 5]{varju2019absolute}
and is therefore omitted.
\begin{lem}
\label{lem:avg ent as dif ent}Let $X$ be a bounded random vector
in $\mathbb{R}^{d}$. Then for all $(r_{1},...,r_{d})=r\in\mathbb{R}_{>0}^{d}$,
we have 
\[
H(X;r)=H(X+I_{r})-\log\det r,
\]
where $I_{r}$ is a uniform random vector in $[0,r_{1}]\times...\times[0,r_{d}]$
independent of $X$.
\end{lem}

From Lemma \ref{lem:avg ent as dif ent} and (\ref{eq:trans =000026 scaling of diff ent})
we get that,

\begin{equation}
H\left(T_{x}\mu;r\right)=H\left(\mu;r\right)\text{ for all }\mu\in\mathcal{M}(\mathbb{R}^{d}),\:r\in\mathbb{R}_{>0}^{d}\text{ and }x\in\mathbb{R}^{d}.\label{eq:avg ent is trans inv}
\end{equation}

When the ratio of scales is non-integral, we have the following version
of (\ref{eq:ub on cond avg ent int ratio}).
\begin{lem}
\label{lem:cond avg ent is pos =000026 ub}Let $X$ be a bounded random
vector in $\mathbb{R}^{d}$, and let $r=(r_{1},...,r_{d})$ and $r'=(r_{1}',...,r_{d}')$
be in $\mathbb{R}_{>0}^{d}$ with $r\le r'$. Then, 
\[
0\le H\left(X;r\mid r'\right)\le\sum_{j=1}^{d}\log\left\lceil r_{j}'/r_{j}\right\rceil .
\]
\end{lem}

\begin{proof}
The proof of the inequality $H\left(X;r\mid r'\right)\ge0$ is similar
to the proof given in \cite[Lemma 8]{varju2019absolute} for the case
$d=1$ and is therefore omitted. In order to prove the second inequality,
set $N:=\left(\left\lceil r_{1}'/r_{1}\right\rceil ,...,\left\lceil r_{d}'/r_{d}\right\rceil \right)$.
By the first inequality and since $Nr\ge r'$, we have $H\left(X;r\mid r'\right)\le H\left(X;r\mid Nr\right)$.
The lemma now follows directly from (\ref{eq:ub on cond avg ent int ratio}).
\end{proof}
The following lemma shows that average entropy is continuous in its
first argument. Its proof is similar to the proof of \cite[Lemma 7]{varju2019absolute}
and is therefore omitted.
\begin{lem}
\label{lem:close rv --> close avg ent}There exists $C>1$, which
depends only on $d$, so that the following holds. Let $R_{1},R_{2}\in\mathbb{R}_{>0}$
be with $R_{1}\le C^{-1}R_{2}$, let $X,Y$ be bounded random vectors
in $\mathbb{R}^{d}$ so that $|X-Y|\le R_{1}$ almost surely, and
let $(r_{1},...,r_{d})=r\in\mathbb{R}_{>0}^{d}$ be with $r_{j}\ge R_{2}$
for $1\le j\le d$. Then,
\[
\left|H\left(X;r\right)-H\left(Y;r\right)\right|\le C\frac{R_{1}}{R_{2}}\log(R_{2}/R_{1}).
\]
\end{lem}

For $r,r'\in\mathbb{R}_{>0}^{d}$ and a nonnegative compactly supported
Borel measure $\theta$ on $\mathbb{R}^{d}$ of total mass $c>0$,
we write $H(\theta;r)$ and $H(\theta;r\mid r')$ in place of $cH(c^{-1}\theta;r)$
and $cH(c^{-1}\theta;r\mid r')$. The proof of the following lemma
is similar to the proof of \cite[Lemma 10]{varju2019absolute} and
is therefore omitted.
\begin{lem}
\label{lem:ent of sum of pos meas >=00003D}Let $\mu_{1},...,\mu_{k}$
be nonnegative compactly supported Borel measures on $\mathbb{R}^{d}$,
and let $r,r'\in\mathbb{R}_{>0}^{d}$ be with $r'/r\in\mathbb{Z}_{>0}^{d}$.
Then,
\[
H\left(\mu_{1}+...+\mu_{k};r\mid r'\right)\ge H\left(\mu_{1};r\mid r'\right)+...+H\left(\mu_{k};r\mid r'\right).
\]
\end{lem}

\subsection{Submodularity inequality and applications}

We shall need the following result from \cite[Theorem 1]{madiman2008entropy}
(see also \cite[Theorem 10]{BV-entropy}).
\begin{thm}[Submodularity inequality]
\label{thm:submodularity}Let $X,Y,Z$ be independent random vectors
in $\mathbb{R}^{d}$. Suppose that $Y$, $X+Y$, $Y+Z$ and $X+Y+Z$
are absolutely continuous with finite differential entropy. Then,
\[
H\left(X+Y+Z\right)+H\left(Y\right)\le H\left(X+Y\right)+H\left(Y+Z\right).
\]

\end{thm}

The following lemma shows that, in the case of integral ratio of scales,
average conditional entropy does not decrease under convolution. Its
proof, which relies on Theorem \ref{thm:submodularity}, is similar
to the proof of \cite[Lemma 6]{varju2019absolute} and is therefore
omitted.
\begin{lem}
\label{lem:conv don't dec}For $\mu,\nu\in\mathcal{M}(\mathbb{R}^{d})$
and $r,r'\in\mathbb{R}_{>0}^{d}$ with $r'/r\in\mathbb{Z}_{>0}^{d}$,
we have $H\left(\mu*\nu;r\mid r'\right)\ge H\left(\mu;r\mid r'\right)$.
\end{lem}

The following statement is a version for average entropy of a classical
lemma due to Kaimanovich and Vershik \cite{kaimanovich_and_vershik}.
Given $\nu\in\mathcal{M}(\mathbb{R}^{d})$ and $k\ge1$, we write
$\nu^{*k}$ for the $k$-fold convolution of $\nu$ with itself.
\begin{lem}
\label{lem:KV by submodularity}Let $\mu,\nu\in\mathcal{M}(\mathbb{R}^{d})$
and $(r_{1},...,r_{d})=r\in\mathbb{R}_{>0}^{d}$ be given. Then for
each $k\ge1$,
\[
H\left(\mu*\nu^{*k};r\right)-H\left(\mu;r\right)\le k\left(H\left(\mu*\nu;r\right)-H\left(\mu;r\right)\right).
\]
\end{lem}

\begin{proof}
For $k\ge1$ set,
\[
\delta_{k}:=H\left(\mu*\nu^{*k};r\right)-H\left(\mu*\nu^{*(k-1)};r\right).
\]
In order to prove the lemma, it suffices to show that $\delta_{k+1}\le\delta_{k}$
for each $k\ge1$. Fix $k\ge1$, let $X$ be a random vector with
distribution $\mu$, let $Y_{1},...,Y_{k+1}$ be random vectors distributed
according to $\nu$, and let $I_{r}$ be a uniform random vector in
$[0,r_{1}]\times...\times[0,r_{d}]$. Suppose that $X,Y_{1},...,Y_{k+1},I_{r}$
are all independent. Setting $Z:=I_{r}+X+Y_{1}+...+Y_{k-1}$, it follows
from Lemma \ref{lem:avg ent as dif ent} that
\[
\delta_{k}=H\left(Z+Y_{k+1}\right)-H\left(Z\right)\text{ and }\delta_{k+1}=H\left(Z+Y_{k}+Y_{k+1}\right)-H\left(Z+Y_{k}\right).
\]
Together with Theorem \ref{thm:submodularity} this gives $\delta_{k+1}\le\delta_{k}$,
which completes the proof of the lemma.
\end{proof}

\section{\label{sec:An-entropy-increase result}An entropy increase result}

The purpose of this section is to prove Theorem \ref{thm:effective ent inc result}.
The general strategy follows the proof of \cite[Theorem 3]{varju2019absolute},
but, as explained in Section \ref{subsec:About-the-proof}, there
are substantial differences in the argument. In Section \ref{subsec:Measure-decomposition},
decompositions of elements of $\mathcal{M}(\mathbb{R}^{d})$ are constructed
in which a non-negligible part of the mass is captured by Bernoulli
measures of a certain scale. In Section \ref{subsec:Entropy-increase-with bernoulli},
entropy increase is obtained for convolutions with Bernoulli measures.
In Section \ref{subsec:Proof-of-ent increase result}, we use these
results in order to complete the proof of Theorem \ref{thm:effective ent inc result}.

\subsection{\label{subsec:Measure-decomposition}Measure decomposition}

The following proposition, which extends \cite[Proposition 21]{varju2019absolute},
is the main result of this subsection. Given a nonnegative measure
$\theta$, we write $\Vert\theta\Vert$ for its total mass. Recall
the sequence $\{s_{n}\}_{n\ge0}\subset\mathbb{R}_{>0}^{d}$ defined
in Section \ref{subsec:Sequences-with-integral ratios}.
\begin{prop}
\label{prop:measure decomp}For all $N\ge N(\lambda)\ge1$ there exists
$\epsilon=\epsilon(\lambda,N)>0$ such that the following holds. Let
$n\ge N$ and let $\nu\in\mathcal{M}(\mathbb{R}^{d})$ be finitely
supported and with,
\[
H\left(\nu;s_{n+N}\mid s_{n}\right)\le\frac{3}{2}H\left(\nu;s_{n}\mid s_{n-N}\right).
\]
Then $\nu=\theta+\zeta_{1}+...+\zeta_{L}$, where $\theta,\zeta_{1},...,\zeta_{L}$
are nonnegative measures,
\[
\Vert\zeta_{1}\Vert+...+\Vert\zeta_{L}\Vert\ge\epsilon\frac{H\left(\nu;s_{n}\mid s_{n-N}\right)}{\max\left\{ 1,-\log H\left(\nu;s_{n}\mid s_{n-N}\right)\right\} },
\]
and for each $1\le i\le L$ there exist $x_{i},y_{i}\in\mathbb{R}^{d}$
so that $\zeta_{i}=\frac{\Vert\zeta_{i}\Vert}{2}(\delta_{x_{i}}+\delta_{y_{i}})$
and $\epsilon\le|s_{n}^{-1}(x_{i}-y_{i})|\le\epsilon^{-1}$.
\end{prop}

The proof of the proposition requires some preparations.
\begin{lem}
\label{lem:ent >=00003D 2 ent}For all $N\ge N(\lambda)\ge1$, $\theta\in\mathcal{M}(\mathbb{R}^{d})$
with $\mathrm{diam}(\mathrm{supp}(\theta))\le1/3$, and $n\ge N$,
we have
\[
H\left(\theta;\frac{s_{n+N}}{s_{n+2N}}\mid\frac{s_{n}}{s_{n+2N}}\right)\ge2H\left(\theta;\frac{s_{n}}{s_{n+2N}}\mid\frac{s_{n-N}}{s_{n+2N}}\right).
\]
\end{lem}

\begin{proof}
Let $\theta\in\mathcal{M}(\mathbb{R}^{d})$ be with $\mathrm{diam}(\mathrm{supp}(\theta))\le1/3$.
By (\ref{eq:avg ent is trans inv}) we may assume that $\mathrm{supp}(\theta)\subset(-1,0)^{d}$.
Let $X=(X_{1},...,X_{d})$ be a random vector with distribution $\theta$.
For $1\le j\le d$ and $t\ge0$ define the events,
\[
E_{j,t}^{-}:=\left\{ X_{j}+t<0\right\} \text{ and }E_{j,t}^{+}:=\left\{ X_{j}+t\ge0\right\} .
\]

Let $(r_{1},...,r_{d})=r\in\mathbb{R}_{>0}^{d}$ be with $r_{j}>1$
for $1\le j\le d$. Given $1\le j\le d$ and $0\le t\le r_{j}$ it
follows from $\mathrm{supp}(\theta)\subset(-1,0)^{d}$ that almost
surely $\left\lfloor (X_{j}+t)/r_{j}\right\rfloor =-1$ on $E_{j,t}^{-}$
and $\left\lfloor (X_{j}+t)/r_{j}\right\rfloor =0$ on $E_{j,t}^{+}$.
Thus,
\begin{eqnarray}
H(\theta;r) & = & \frac{1}{\det r}\int_{\times_{j=1}^{d}[0,r_{j})}H\left(\left\lfloor (X+y)/r\right\rfloor \right)\:dy\nonumber \\
 & = & \frac{1}{\det r}\int_{\times_{j=1}^{d}[0,r_{j})}\sum_{u\in\{-,+\}^{d}}F\left(\mathbb{P}\left(\cap_{j=1}^{d}E_{j,y_{j}}^{u_{j}}\right)\right)\:dy,\label{eq:H(theta,D)=00003Dint F}
\end{eqnarray}
where recall that $F(t)=-t\log t$.

For $1\le j\le d$ set $I_{0}^{j}:=[0,1)$ and $I_{1}^{j}:=[1,r_{j})$.
Note that,
\begin{equation}
\mathbb{P}\left(E_{j,t}^{+}\right)=1\text{ for }1\le j\le d\text{ and }t\in I_{1}^{j}.\label{eq:P=00003D1 for t in I_1}
\end{equation}
Write $\overline{\mathbf{1}}$ in place of $(1,...,1)\in\{0,1\}^{d}$.
For $(v_{1},...,v_{d})=v\in\{0,1\}^{d}\setminus\{\overline{\mathbf{1}}\}$
set $l_{v}:=\#\left\{ 1\le j\le d\::\:v_{j}=0\right\} $, let $1\le j_{v,1}<...<j_{v,l_{v}}\le d$
be with $v_{j_{v,k}}=0$ for $1\le k\le l_{v}$, set $J_{v}:=[d]\setminus\{j_{v,1},...,j_{v,l_{v}}\}$,
and write
\[
C_{v}:=\int_{[0,1)^{l_{v}}}\sum_{u\in\{-,+\}^{l_{v}}}F\left(\mathbb{P}\left(\cap_{k=1}^{l_{v}}E_{j_{v,k},y_{k}}^{u_{k}}\right)\right)\:dy_{1}...dy_{l_{v}}.
\]
Note that $C_{v}$ does not depend on $r$. By (\ref{eq:P=00003D1 for t in I_1})
\[
\frac{1}{\det r}\int_{\times_{j=1}^{d}I_{v_{j}}^{j}}\sum_{u\in\{-,+\}^{d}}F\left(\mathbb{P}\left(\cap_{j=1}^{d}E_{j,y_{j}}^{u_{j}}\right)\right)\:dy=C_{v}\prod_{1\le k\le l_{v}}r_{j_{v,k}}^{-1}\prod_{j\in J_{v}}(1-r_{j}^{-1}),
\]
and
\[
\frac{1}{\det r}\int_{\times_{j=1}^{d}I_{1}^{j}}\sum_{u\in\{-,+\}^{d}}F\left(\mathbb{P}\left(\cap_{j=1}^{d}E_{j,y_{j}}^{u_{j}}\right)\right)\:dy=0.
\]
Thus, from (\ref{eq:H(theta,D)=00003Dint F})
\begin{equation}
H(\theta;r)=\sum_{v\in\{0,1\}^{d}\setminus\{\overline{\mathbf{1}}\}}C_{v}\prod_{1\le k\le l_{v}}r_{j_{v,k}}^{-1}\prod_{j\in J_{v}}(1-r_{j}^{-1}).\label{eq:H(theta,D)=00003D sum C_v}
\end{equation}

Recall that by (\ref{eq:prop of sequences}) we have $\lambda_{j}^{n}\le s_{n,j}<2\lambda_{j}^{n}$
for $1\le j\le d$ and $n\ge0$. From this and since (\ref{eq:H(theta,D)=00003D sum C_v})
holds for all $(r_{1},...,r_{d})=r\in\mathbb{R}_{>0}^{d}$ with $r_{j}>1$
for $1\le j\le d$, it follows that for all sufficiently large $N\ge1$
and all $n\ge0$,
\[
H\left(\theta;\frac{s_{n}}{s_{n+N}}\right)=\Theta\left(\sum_{v\in\{0,1\}^{d}\setminus\{\overline{\mathbf{1}}\}}C_{v}\prod_{1\le k\le l_{v}}\lambda_{j_{v,k}}^{N}\right).
\]
Now the lemma follows directly by choosing $N$ to be sufficiently
large.
\end{proof}
The proof of the following lemma is similar to the proof of \cite[Lemma 25]{varju2019absolute}
and is therefore omitted. Recall that we write $|r|$ for the Euclidean
norm of $r\in\mathbb{R}_{>0}^{d}$.
\begin{lem}
\label{lem:ent of theta =00003D sum of ent of rest}Let $r,r'\in\mathbb{R}_{>0}^{d}$
and let $B_{1},...,B_{k}\subset\mathbb{R}^{d}$ be closed balls such
that $\mathrm{dist}(B_{i_{1}},B_{i_{2}})>|r|,|r'|$ for all $1\le i_{1}<i_{2}\le k$.
Then $H\left(\theta;r\mid r'\right)=\sum_{i=1}^{k}H\left(\theta|_{B_{i}};r\mid r'\right)$
for all $\theta\in\mathcal{M}\left(\cup_{i=1}^{k}B_{i}\right)$.
\end{lem}

The proof of the following lemma is similar to the proof of \cite[Lemma 26]{varju2019absolute}
and is therefore omitted.
\begin{lem}
\label{lem:ent of nu =00003D theta + eta}Let $\nu\in\mathcal{M}(\mathbb{R}^{d})$
and let $\theta,\zeta$ be nonnegative Borel measures with $\nu=\theta+\zeta$
and $\Vert\zeta\Vert<1/2$. Then for all $r\in\mathbb{R}_{>0}^{d}$
and $N\in\mathbb{Z}_{>0}^{d}$,
\[
H\left(\theta;r\mid Nr\right)\le H\left(\nu;r\mid Nr\right)\le H\left(\theta;r\mid Nr\right)+2\Vert\zeta\Vert\log\left(\Vert\zeta\Vert^{-1}\det N\right).
\]
\end{lem}

\begin{proof}[Proof of Proposition \ref{prop:measure decomp}]
Let $N\ge1$ be large with respect to $\lambda$, let $n\ge N$,
and let $\nu\in\mathcal{M}(\mathbb{R}^{d})$ be finitely supported
and with
\begin{equation}
H\left(\nu;s_{n+N}\mid s_{n}\right)\le\frac{3}{2}H\left(\nu;s_{n}\mid s_{n-N}\right).\label{eq:assump on nu}
\end{equation}

Set $R:=2\left|\lambda^{-3N}\right|$. Let $S_{s_{n+2N}}^{-1}\nu=\theta+\zeta_{1}+...+\zeta_{L}$
be a decomposition of $S_{s_{n+2N}}^{-1}\nu$ such that $\theta,\zeta_{1},...,\zeta_{L}$
are nonnegative measures, for each $1\le i\le L$ there exist $x_{i},y_{i}\in\mathbb{R}^{d}$
so that $\zeta_{i}=\frac{\Vert\zeta_{i}\Vert}{2}(\delta_{x_{i}}+\delta_{y_{i}})$
and $1/6\le|x_{i}-y_{i}|\le R$, and $\Vert\theta\Vert$ is minimal
among all such decompositions. Since $\nu$ is finitely supported,
the minimum clearly exists.

Let $0<\epsilon<1$ be small with respect to $\lambda$ and $N$,
and set $\delta:=1-\Vert\theta\Vert$. By (\ref{eq:prop of sequences})
we have $\lambda_{d}^{2N}/12\le\left|s_{n}^{-1}s_{n+2N}\left(x_{i}-y_{i}\right)\right|\le R$
for $1\le i\le L$. Thus, in order to prove the proposition it suffices
to show that
\begin{equation}
\delta\ge\epsilon\frac{H\left(\nu;s_{n}\mid s_{n-N}\right)}{\max\left\{ 1,-\log H\left(\nu;s_{n}\mid s_{n-N}\right)\right\} }.\label{eq:delta>=00003Depsilon times frac}
\end{equation}
From (\ref{eq:prop of sequences}) and (\ref{eq:ub on cond avg ent int ratio}),
\[
H\left(\nu;s_{n}\mid s_{n-N}\right)\le\det\left(s_{n-N}/s_{n}\right)\le2^{d}\lambda_{d}^{-dN}.
\]
Hence we may assume that $\delta<\epsilon^{1/2}$, otherwise (\ref{eq:delta>=00003Depsilon times frac})
holds whenever $\epsilon$ is sufficiently small with respect to $\lambda$
and $N$.

By the minimality of $\Vert\theta\Vert$, for each $x,y\in\mathrm{supp}(\theta)$
we have $|x-y|<1/6$ or $|x-y|>R$. Thus, it is easy to see that there
exist closed balls $B_{1},...,B_{k}\subset\mathbb{R}^{d}$ so that
$\mathrm{diam}(B_{i})<1/3$ for $1\le i\le k$, $\mathrm{dist}(B_{i_{1}},B_{i_{2}})>R$
for $1\le i_{1}<i_{2}\le k$, and $\mathrm{supp}(\theta)\subset\cup_{i=1}^{k}B_{i}$.
By (\ref{eq:prop of sequences}) and the choice of $R$,
\[
\left|\frac{s_{n+N}}{s_{n+2N}}\right|,\left|\frac{s_{n}}{s_{n+2N}}\right|,\left|\frac{s_{n-N}}{s_{n+2N}}\right|\le R.
\]
Thus, by Lemmata \ref{lem:ent >=00003D 2 ent} and \ref{lem:ent of theta =00003D sum of ent of rest},
\[
H\left(\theta;\frac{s_{n+N}}{s_{n+2N}}\mid\frac{s_{n}}{s_{n+2N}}\right)\ge2H\left(\theta;\frac{s_{n}}{s_{n+2N}}\mid\frac{s_{n-N}}{s_{n+2N}}\right).
\]
Together with (\ref{eq:scaling rel of avg ent}) this gives,
\begin{equation}
H\left(S_{s_{n+2N}}\theta;s_{n+N}\mid s_{n}\right)\ge2H\left(S_{s_{n+2N}}\theta;s_{n}\mid s_{n-N}\right).\label{eq:ent of push of theta}
\end{equation}

From Lemma \ref{lem:ent of nu =00003D theta + eta}, $\delta<\epsilon^{1/2}<1/2$
and $\det\left(s_{n-N}/s_{n}\right)\le2^{d}\lambda_{d}^{-dN}$, we
get
\[
H\left(\nu;s_{n+N}\mid s_{n}\right)\ge H\left(S_{s_{n+2N}}\theta;s_{n+N}\mid s_{n}\right)
\]
and
\[
H\left(S_{s_{n+2N}}\theta;s_{n}\mid s_{n-N}\right)\ge H\left(\nu;s_{n}\mid s_{n-N}\right)-2\delta\log\left(\delta^{-1}2^{d}\lambda_{d}^{-dN}\right).
\]
Together with (\ref{eq:ent of push of theta}) and $\delta<\epsilon^{1/2}<2^{-d}\lambda_{d}^{dN}$,
these inequalities yield
\[
H\left(\nu;s_{n+N}\mid s_{n}\right)\ge2H\left(\nu;s_{n}\mid s_{n-N}\right)+8\delta\log\delta.
\]
Hence, by (\ref{eq:assump on nu}),
\begin{equation}
\delta\log\delta^{-1}\ge\frac{1}{16}H\left(\nu;s_{n}\mid s_{n-N}\right).\label{eq:delta log delta^-1}
\end{equation}

Now assume by contradiction that (\ref{eq:delta>=00003Depsilon times frac})
does not hold, and set $h:=H\left(\nu;s_{n}\mid s_{n-N}\right)$.
From $\delta<\epsilon^{1/2}$ and (\ref{eq:delta log delta^-1}),
and by assuming that $\epsilon$ is sufficiently small, we get $h<1/2$.
Thus, since (\ref{eq:delta>=00003Depsilon times frac}) does not hold,
we have $\delta<\frac{\epsilon h}{-\log h}$. From this, since $t\rightarrow t\log t^{-1}$
is increasing on a right neighbourhood of $0$, and by assuming that
$\epsilon$ is sufficiently small, we obtain
\[
\delta\log\delta^{-1}<\frac{\epsilon h}{-\log h}\log\left(\epsilon^{-1}h^{-1}\log h^{-1}\right)\le h/16.
\]
But this contradicts (\ref{eq:delta log delta^-1}), which completes
the proof of the proposition.
\end{proof}

\subsection{\label{subsec:Entropy-increase-with bernoulli}Entropy increase for
convolutions with Bernoulli measures}

The following proposition is the main result of this subsection. Recall
that the notion of an $(\epsilon,m)$-non-saturated measure is defined
in Definition \ref{def:non-saturated-measure}.
\begin{prop}
\label{prop:ent inc bernoulli}For each $0<\epsilon<1$ and $M\ge1$
there exists $\delta=\delta(\lambda,\epsilon,M)>0$ such that for
all $N\ge N(\lambda,\epsilon,M)\ge1$ and $n\ge N$ the following
holds. Let $\mu\in\mathcal{M}(\mathbb{R}^{d})$ be $(\epsilon,m)$-non-saturated
for all $m\ge M$, let $x,y\in\mathbb{R}^{d}$ be with $\epsilon\le|s_{n}^{-1}(x-y)|\le\epsilon^{-1}$,
and set $\zeta:=\frac{1}{2}(\delta_{x}+\delta_{y})$. Then,
\[
H\left(\mu*\zeta;s_{n}\mid s_{n-N}\right)\ge H\left(\mu;s_{n}\mid s_{n-N}\right)+\delta.
\]
\end{prop}

\subsubsection{Entropy of repeated self-convolutions of Bernoulli measures}

For $1\le j\le d$ and $\sigma\in\mathcal{M}(e_{j}\mathbb{R})$, we
write $\mathrm{Var}(\sigma)$ to denote the variance of $\sigma'$,
where $\sigma'\in\mathcal{M}(\mathbb{R})$ is the push-forward of
$\sigma$ via the map sending $te_{j}$ to $t$. Recall that in Section
\ref{subsec:Basic-notations} we describe how we use the relation
$\ll$. The following lemma, whose proof relies on the Berry-Esseen
theorem, is established in \cite[Lemma 3.2]{Rap_SA_diag}.
\begin{lem}
\label{lem:ent of slices of O(1) measures}Let $0<\epsilon<1$ and
$m,l,k\in\mathbb{Z}_{>0}$ be with $\epsilon^{-1}\ll m\ll l\ll k$,
and let $\theta_{1},...,\theta_{k}\in\mathcal{M}(\mathbb{R}^{d})$
be with $\mathrm{diam}(\mathrm{supp}(\theta_{i}))\le\epsilon^{-1}$
for $1\le i\le k$. Set $\sigma:=\theta_{1}*...*\theta_{k}$, and
suppose that there exists $1\le j\le d$ so that $\mathrm{Var}(\pi_{j}\sigma)\ge\epsilon k$
and $\mathrm{Var}(\pi_{j'}\sigma)\le\epsilon^{-1}$ for $1\le j'<j$.
Then for $a:=\left\lfloor \frac{1}{2\chi_{j}}\log k\right\rfloor $,
\[
\frac{1}{m}H\left(\sigma,\mathcal{E}_{l-a+m}\mid\mathcal{E}_{l-a}\vee\pi_{[d]\setminus\{j\}}^{-1}\mathcal{E}_{l-a+m}\right)>\chi_{j}-\epsilon.
\]
\end{lem}

\begin{cor}
\label{cor:exists j s.t. saturated}Let $0<\epsilon<1$ and $m_{1},...,m_{d},l_{1},...,l_{d},k_{1},...,k_{d}\in\mathbb{Z}_{>0}$
be such that $\epsilon^{-1}\ll m_{d}$, $m_{j}\ll l_{j}\ll k_{j}$
for $1\le j\le d$, and $k_{j+1}\ll m_{j}$ for $1\le j<d$. Let $x,y\in\mathbb{R}^{d}$
be with $\epsilon\le|x-y|\le\epsilon^{-1}$, and set $\zeta:=\frac{1}{2}(\delta_{x}+\delta_{y})$.
Then there exists $1\le j\le d$ so that for $a:=\left\lfloor \frac{1}{2\chi_{j}}\log k_{j}\right\rfloor $,
\[
\frac{1}{m_{j}}H\left(\zeta^{*k_{j}},\mathcal{E}_{l_{j}-a+m_{j}}\mid\mathcal{E}_{l_{j}-a}\vee\pi_{[d]\setminus\{j\}}^{-1}\mathcal{E}_{l_{j}-a+m_{j}}\right)>\chi_{j}-\epsilon.
\]
\end{cor}

\begin{proof}
Set $\epsilon_{d}:=\epsilon^{2}/(4d)$ and for $1\le j<d$ set $\epsilon_{j}:=k_{j+1}^{-1}$.
We have,
\[
4\sum_{j=1}^{d}\mathrm{Var}(\pi_{j}\zeta)=\sum_{j=1}^{d}\left|\pi_{j}x-\pi_{j}y\right|^{2}=|x-y|^{2}\ge\epsilon^{2}.
\]
Thus, there exists $1\le j\le d$, which we fix, so that $\mathrm{Var}(\pi_{j}\zeta)\ge\epsilon_{j}$
and $\mathrm{Var}(\pi_{j'}\zeta)<\epsilon_{j'}$ for $1\le j'<j$.
Setting $\sigma:=\zeta^{*k_{j}}$, it holds that $\mathrm{Var}(\pi_{j'}\sigma)=k_{j}\mathrm{Var}(\pi_{j'}\zeta)$
for $1\le j'\le d$. Hence, $\mathrm{Var}(\pi_{j}\sigma)\ge k_{j}\epsilon_{j}$
and $\mathrm{Var}(\pi_{j'}\sigma)<k_{j}\epsilon_{j'}\le1$ for $1\le j'<j$.
Now the lemma follows directly from Lemma \ref{lem:ent of slices of O(1) measures}.
\end{proof}

\subsubsection{Proof of the proposition}
\begin{proof}[Proof of Proposition \ref{prop:ent inc bernoulli}]
Let $0<\epsilon<1$ and $M,N,n\in\mathbb{Z}_{>0}$ be with $\epsilon^{-1},M\ll N$
and $n\ge N$, let $\mu\in\mathcal{M}(\mathbb{R}^{d})$ be $(\epsilon,m)$-non-saturated
for all $m\ge M$, let $x,y\in\mathbb{R}^{d}$ be with $\epsilon\le|s_{n}^{-1}(x-y)|\le\epsilon^{-1}$,
and set $\zeta:=\frac{1}{2}(\delta_{x}+\delta_{y})$.

Let $m_{1},...,m_{d},l_{1},...,l_{d},k_{1},...,k_{d}\in\mathbb{Z}_{>0}$
be such that $\epsilon^{-1},M\ll m_{d}$, $k_{1}\ll N$, $m_{j}\ll l_{j}\ll k_{j}$
for $1\le j\le d$, and $k_{j+1}\ll m_{j}$ for $1\le j<d$. By Corollary
\ref{cor:exists j s.t. saturated} there exists $1\le j\le d$ so
that for $a:=\left\lfloor \frac{1}{2\chi_{j}}\log k_{j}\right\rfloor $,
\[
\frac{1}{m_{j}}H\left(\left(S_{s_{n}}^{-1}\zeta\right)^{*k_{j}},\mathcal{E}_{l_{j}-a+m_{j}}\mid\mathcal{E}_{l_{j}-a}\vee\pi_{[d]\setminus\{j\}}^{-1}\mathcal{E}_{l_{j}-a+m_{j}}\right)>\chi_{j}-\epsilon/4.
\]
Thus, setting $n':=n+l_{j}-a$ and $\mathcal{C}:=\mathcal{E}_{n'}\vee\pi_{[d]\setminus\{j\}}^{-1}\mathcal{E}_{n'+m_{j}}$,
it follows from $\epsilon^{-1}\ll m_{d}$ and (\ref{eq:commens part by s_n})
that
\[
\frac{1}{m_{j}}H\left(\zeta^{*k_{j}},\mathcal{E}_{n'+m_{j}}\mid\mathcal{C}\right)>\chi_{j}-\epsilon/3.
\]
From this, since $\mu$ is $(\epsilon,m_{j})$-non-saturated, and
by (\ref{eq:commens part by T_x}) combined with the concavity of
conditional entropy,
\[
\frac{1}{m_{j}}H\left(\mu*\zeta^{*k_{j}},\mathcal{E}_{n'+m_{j}}\mid\mathcal{C}\right)>\frac{1}{m_{j}}H\left(\mu,\mathcal{E}_{n'+m_{j}}\mid\mathcal{C}\right)+\epsilon/2.
\]

By (\ref{eq:commens part by T_x}) and concavity it also follows that,
\[
\frac{1}{m_{j}}H\left(\mu*\zeta^{*k_{j}},\pi_{[d]\setminus\{j\}}^{-1}\mathcal{E}_{n'+m_{j}}\mid\mathcal{E}_{n'}\right)\ge\frac{1}{m_{j}}H\left(\mu,\pi_{[d]\setminus\{j\}}^{-1}\mathcal{E}_{n'+m_{j}}\mid\mathcal{E}_{n'}\right)-\epsilon/6.
\]
From the last two inequalities together with the conditional entropy
formula,
\[
\frac{1}{m_{j}}H\left(\mu*\zeta^{*k_{j}},\mathcal{E}_{n'+m_{j}}\mid\mathcal{E}_{n'}\right)\ge\frac{1}{m_{j}}H\left(\mu,\mathcal{E}_{n'+m_{j}}\mid\mathcal{E}_{n'}\right)+\epsilon/3.
\]
Since $m_{j},l_{j}\ll k_{j}$, we have $n'+m_{j}\le n$. Thus, by
the last inequality and concavity,
\begin{multline*}
H\left(\mu*\zeta^{*k_{j}},\mathcal{E}_{n}\right)\\
=H\left(\mu*\zeta^{*k_{j}},\mathcal{E}_{n}\mid\mathcal{E}_{n'+m_{j}}\right)+H\left(\mu*\zeta^{*k_{j}},\mathcal{E}_{n'+m_{j}}\mid\mathcal{E}_{n'}\right)+H\left(\mu*\zeta^{*k_{j}},\mathcal{E}_{n'}\right)\\
\ge H\left(\mu,\mathcal{E}_{n}\mid\mathcal{E}_{n'+m_{j}}\right)+H\left(\mu,\mathcal{E}_{n'+m_{j}}\mid\mathcal{E}_{n'}\right)+H\left(\mu,\mathcal{E}_{n'}\right)+(m_{j}\epsilon)/3-O(1)\\
\ge H\left(\mu,\mathcal{E}_{n}\right)+(m_{j}\epsilon)/4.
\end{multline*}
Together with (\ref{eq:two defs of ent are equiv}) this gives,
\[
H\left(\mu*\zeta^{*k_{j}},s_{n}\right)\ge H\left(\mu,s_{n}\right)+(m_{j}\epsilon)/5.
\]
Hence by Lemma \ref{lem:KV by submodularity},
\begin{equation}
H\left(\mu*\zeta,s_{n}\right)\ge H\left(\mu,s_{n}\right)+\frac{m_{j}\epsilon}{5k_{j}}.\label{eq:lb no conditioning}
\end{equation}

Let $X$ and $Z$ be independent random vectors with distributions
$S_{s_{n-N}}^{-1}\mu$ and $S_{s_{n-N}}^{-1}\zeta$ respectively,
and set $Y:=X+Z$. From $|s_{n}^{-1}(x-y)|\le\epsilon^{-1}$ and (\ref{eq:prop of sequences})
it follows that $\left|s_{n-N}^{-1}x+X-Y\right|\le2d\epsilon^{-1}\lambda_{1}^{N}$
almost surly. Thus, by Lemma \ref{lem:close rv --> close avg ent}
and (\ref{eq:avg ent is trans inv}),
\[
\left|H\left(X;1_{\mathbb{R}_{>0}^{d}}\right)-H\left(Y;1_{\mathbb{R}_{>0}^{d}}\right)\right|=O\left(-\epsilon^{-1}\lambda_{1}^{N}\log\left(\epsilon^{-1}\lambda_{1}^{N}\right)\right),
\]
where $1_{\mathbb{R}_{>0}^{d}}$ is the identity of $\mathbb{R}_{>0}^{d}$.
From this, by (\ref{eq:scaling rel of avg ent}) and since $\epsilon^{-1},m_{j},k_{j}\ll N$,
we may assume that
\[
\left|H\left(\mu;s_{n-N}\right)-H\left(\mu*\zeta;s_{n-N}\right)\right|\le\frac{m_{j}\epsilon}{10k_{j}}.
\]
By combining this with (\ref{eq:lb no conditioning}) we get
\[
H\left(\mu*\zeta,s_{n}\mid s_{n-N}\right)\ge H\left(\mu,s_{n}\mid s_{n-N}\right)+\frac{m_{j}\epsilon}{10k_{j}},
\]
which completes the proof of the proposition with $\delta:=(m_{j}\epsilon)/(10k_{j})$.
\end{proof}

\subsection{\label{subsec:Proof-of-ent increase result}Proof of Theorem \ref{thm:effective ent inc result}}

For the proof of the theorem we shall need the following lemma.
\begin{lem}
\label{lem:prep for effct ent inc thm}For all $N\ge1$ there exists
$C=C(\lambda,N)>1$ such that the following holds. Let $\nu\in\mathcal{M}(\mathbb{R}^{d})$,
$0<\beta<1/2$ and $t_{2}>t_{1}>0$ be with $t_{2}-t_{1}>3N-\frac{1}{\log\lambda_{1}}$
and $\frac{1}{t_{2}-t_{1}}H(\nu;\lambda^{t_{2}}\mid\lambda^{t_{1}})>\beta$.
Set $\tau_{1}:=\left\lceil \frac{t_{1}}{N}+1-\frac{1}{N\log\lambda_{1}}\right\rceil $
and $\tau_{2}:=\left\lfloor \frac{t_{2}}{N}\right\rfloor $, and let
$\mathcal{N}$ be the set of all integers $\tau_{1}\le n<\tau_{2}$
so that,
\[
H\left(\nu;s_{Nn}\mid s_{Nn-N}\right)\ge\max\left\{ \frac{2}{3}H\left(\nu;s_{Nn+N}\mid s_{Nn}\right),\frac{\beta}{6}\right\} .
\]
Then,
\[
\sum_{n\in\mathcal{N}}H(\nu;s_{Nn}\mid s_{Nn-N})\ge\frac{\beta}{6}(t_{2}-t_{1})-C.
\]
\end{lem}

\begin{proof}
Let $N\ge1$ be given, and let $\nu,\beta,t_{1},t_{2},\tau_{1},\tau_{2}$
and $\mathcal{N}$ be as in the statement of the lemma. For each $\tau_{1}\le n\le\tau_{2}$
set $\alpha_{n}:=H\left(\nu;s_{Nn}\mid s_{Nn-N}\right)$. Let $\mathcal{N}'$
be the set of all integers $\tau_{1}\le n<\tau_{2}$ so that $\alpha_{n}\ge\frac{2}{3}\alpha_{n+1}$.
Let $n_{1}<n_{2}<...<n_{k-1}$ be an enumeration of $\mathcal{N}'$,
and set $n_{0}:=\tau_{1}-1$ and $n_{k}:=\tau_{2}$. For $1\le q\le k$
and $n_{q-1}<n<n_{q}$ we have $\alpha_{n}\le\frac{2}{3}\alpha_{n+1}$,
which by induction gives $\alpha_{n}\le(2/3)^{n_{q}-n}\alpha_{n_{q}}$.
Thus,
\[
\sum_{n=n_{q-1}+1}^{n_{q}}\alpha_{n}\le3\alpha_{n_{q}}\text{ for }1\le q\le k.
\]

By (\ref{eq:prop of sequences}) and Lemma \ref{lem:cond avg ent is pos =000026 ub},
\[
H\left(\nu;s_{N\tau_{1}-N}\mid\lambda^{t_{1}}\right),H\left(\nu;\lambda^{t_{2}}\mid s_{N\tau_{2}}\right),\alpha_{n_{k}}=O_{\lambda,N}(1).
\]
Additionally,
\[
H(\nu;\lambda^{t_{2}}\mid\lambda^{t_{1}})=H\left(\nu;\lambda^{t_{2}}\mid s_{N\tau_{2}}\right)+H\left(\nu;s_{N\tau_{1}-N}\mid\lambda^{t_{1}}\right)+\sum_{q=1}^{k}\sum_{n=n_{q-1}+1}^{n_{q}}\alpha_{n}.
\]

By combining all of this together with $\frac{1}{t_{2}-t_{1}}H(\nu;\lambda^{t_{2}}\mid\lambda^{t_{1}})>\beta$,
we obtain
\[
\beta(t_{2}-t_{1})<3\sum_{n\in\mathcal{N}'}\alpha_{n}+O_{\lambda,N}(1).
\]
Moreover, by the definition of $\mathcal{N}$,
\[
\sum_{n\in\mathcal{N}'\setminus\mathcal{N}}\alpha_{n}\le\frac{\beta}{6}|\mathcal{N}'|\le\frac{\beta}{6}(t_{2}-t_{1}).
\]
Now, from the last two inequalities
\[
\frac{1}{2}\beta(t_{2}-t_{1})<3\sum_{n\in\mathcal{N}}\alpha_{n}+O_{\lambda,N}(1),
\]
which completes the proof of the lemma. 
\end{proof}
We can now prove Theorem \ref{thm:effective ent inc result}, which
is the following statement.
\begin{thm*}
For each $\epsilon>0$ and $M\ge1$, there exists $C=C(\lambda,\epsilon,M)>1$
such that the following holds. Let $\mu\in\mathcal{M}(\mathbb{R}^{d})$
be $(\epsilon,m)$-non-saturated for all $m\ge M$, and let $\nu\in\mathcal{M}(\mathbb{R}^{d})$,
$0<\beta<1/2$ and $t_{2}>t_{1}>0$ be with $\frac{1}{t_{2}-t_{1}}H(\nu;\lambda^{t_{2}}\mid\lambda^{t_{1}})>\beta$.
Then,
\[
H\left(\nu*\mu;\lambda^{t_{2}}\mid\lambda^{t_{1}}\right)\ge H\left(\mu;\lambda^{t_{2}}\mid\lambda^{t_{1}}\right)+C^{-1}\beta\left(\log\beta^{-1}\right)^{-1}(t_{2}-t_{1})-C.
\]
\end{thm*}
\begin{proof}
Let $0<\epsilon<1$ and $M\ge1$ be given, and let $N_{1},N_{2}\in\mathbb{Z}_{>0}$
and $0<\delta<1$ be with $N_{2}/N_{1}\in\mathbb{Z}$ and
\begin{equation}
\epsilon^{-1},M\ll N_{1}\ll\delta^{-1}\ll N_{2}.\label{eq:rel bet param}
\end{equation}
Suppose also that $N_{1}$ is large with respect to $\lambda$. Let
$\mu,\nu,\beta,t_{1},t_{2}$ be as in the statement of the theorem.

By Lemma \ref{lem:close rv --> close avg ent} we may assume that
$\nu$ is finitely supported. Additionally, by Lemma \ref{lem:cond avg ent is pos =000026 ub}
we may assume that $t_{2}-t_{1}>3N_{1}-\frac{1}{\log\lambda_{1}}$.
Set $\tau_{1}:=\left\lceil \frac{t_{1}}{N_{1}}+1-\frac{1}{N_{1}\log\lambda_{1}}\right\rceil $
and $\tau_{2}:=\left\lfloor \frac{t_{2}}{N_{1}}\right\rfloor $, and
let $\mathcal{N}$ be the set of all integers $\tau_{1}\le n<\tau_{2}$
so that,
\[
H\left(\nu;s_{N_{1}n}\mid s_{N_{1}n-N_{1}}\right)\ge\max\left\{ \frac{2}{3}H\left(\nu;s_{N_{1}n+N_{1}}\mid s_{N_{1}n}\right),\frac{\beta}{6}\right\} .
\]
For an integer $0\le q<N_{2}/N_{1}$ let $\mathcal{N}_{q}$ be the
set of all $n\in\mathcal{N}$ so that $n\ge\tau_{1}+N_{2}/N_{1}$
and $n=q\mod N_{2}/N_{1}$. By Lemma \ref{lem:prep for effct ent inc thm}
there exists $0\le q<N_{2}/N_{1}$, which we fix for the rest of the
proof, so that
\begin{equation}
\sum_{n\in\mathcal{N}_{q}}H\left(\nu;s_{N_{1}n}\mid s_{N_{1}n-N_{1}}\right)\ge\frac{N_{1}\beta}{6N_{2}}(t_{2}-t_{1})-O_{\lambda,N_{1},N_{2}}(1).\label{eq:sum over n in N_q >=00003D}
\end{equation}

Let $n\in\mathcal{N}_{q}$ be given. Since $\nu$ is finitely supported
and by Proposition \ref{prop:measure decomp}, there exist nonnegative
measures $\theta_{n},\zeta_{n,1},...,\zeta_{n,L_{n}}$ so that $\nu=\theta_{n}+\zeta_{n,1}+...+\zeta_{n,L_{n}}$,
\begin{eqnarray}
\sum_{i=1}^{L_{n}}\Vert\zeta_{n,i}\Vert & \ge & \delta\frac{H\left(\nu;s_{N_{1}n}\mid s_{N_{1}n-N_{1}}\right)}{\max\left\{ 1,-\log H\left(\nu;s_{N_{1}n}\mid s_{N_{1}n-N_{1}}\right)\right\} }\nonumber \\
 & \ge & \delta\frac{H\left(\nu;s_{N_{1}n}\mid s_{N_{1}n-N_{1}}\right)}{-4\log\beta},\label{eq:lb on mass of etas}
\end{eqnarray}
and for each $1\le i\le L_{n}$ there exist $x_{n,i},y_{n,i}\in\mathbb{R}^{d}$
with $\zeta_{n,i}=\frac{\Vert\zeta_{n,i}\Vert}{2}(\delta_{x_{n,i}}+\delta_{y_{n,i}})$
and $\delta\le|s_{N_{1}n}^{-1}(x_{n,i}-y_{n,i})|\le\delta^{-1}$.
Since $\mu$ is $(\epsilon,m)$-non-saturated for $m\ge M$, from
(\ref{eq:rel bet param}), and by Proposition \ref{prop:ent inc bernoulli},
for $1\le i\le L_{n}$ we have
\[
H\left(\mu*\zeta_{n,i};s_{N_{1}n}\mid s_{N_{1}n-N_{2}}\right)\ge\Vert\zeta_{n,i}\Vert\left(H\left(\mu;s_{N_{1}n}\mid s_{N_{1}n-N_{2}}\right)+N_{2}^{-1}\right).
\]
From this, by Lemmata \ref{lem:ent of sum of pos meas >=00003D} and
\ref{lem:conv don't dec}, and from (\ref{eq:lb on mass of etas}),
\begin{equation}
H\left(\mu*\nu;s_{N_{1}n}\mid s_{N_{1}n-N_{2}}\right)\ge H\left(\mu;s_{N_{1}n}\mid s_{N_{1}n-N_{2}}\right)+\delta\frac{H\left(\nu;s_{N_{1}n}\mid s_{N_{1}n-N_{1}}\right)}{-4N_{2}\log\beta}.\label{eq:holds for all n in N_q}
\end{equation}

Let $n_{1}<...<n_{k}$ be an enumeration of $\mathcal{N}_{q}$. By
the definition of $\mathcal{N}_{q}$ we have $n_{1}-N_{2}/N_{1}\ge\tau_{1}$,
$n_{k}<\tau_{2}$ and $n_{i+1}-N_{2}/N_{1}\ge n_{i}$ for $1\le i<k$.
Thus,
\begin{multline*}
H\left(\nu*\mu;s_{N_{1}\tau_{2}}\mid s_{N_{1}\tau_{1}}\right)=H\left(\nu*\mu;s_{N_{1}\tau_{2}}\mid s_{N_{1}n_{k}}\right)+H\left(\nu*\mu;s_{N_{1}n_{1}-N_{2}}\mid s_{N_{1}\tau_{1}}\right)\\
+\sum_{i=1}^{k}H\left(\nu*\mu;s_{N_{1}n_{i}}\mid s_{N_{1}n_{i}-N_{2}}\right)+\sum_{i=1}^{k-1}H\left(\nu*\mu;s_{N_{1}n_{i+1}-N_{2}}\mid s_{N_{1}n_{i}}\right).
\end{multline*}
From this, since (\ref{eq:holds for all n in N_q}) holds for all
$n\in\mathcal{N}_{q}$, and by Lemma \ref{lem:conv don't dec},
\[
H\left(\nu*\mu;s_{N_{1}\tau_{2}}\mid s_{N_{1}\tau_{1}}\right)\ge H\left(\mu;s_{N_{1}\tau_{2}}\mid s_{N_{1}\tau_{1}}\right)+\sum_{n\in\mathcal{N}_{q}}\delta\frac{H\left(\nu;s_{N_{1}n}\mid s_{N_{1}n-N_{1}}\right)}{-4N_{2}\log\beta}.
\]
Together with (\ref{eq:sum over n in N_q >=00003D}), the last inequality
gives
\begin{multline*}
H\left(\nu*\mu;s_{N_{1}\tau_{2}}\mid s_{N_{1}\tau_{1}}\right)\ge H\left(\mu;s_{N_{1}\tau_{2}}\mid s_{N_{1}\tau_{1}}\right)\\
+\frac{\delta N_{1}\beta}{24N_{2}^{2}\log\beta^{-1}}(t_{2}-t_{1})-O_{\lambda,N_{1},N_{2}}(1).
\end{multline*}
This, together with (\ref{eq:prop of sequences}) and Lemma \ref{lem:cond avg ent is pos =000026 ub},
completes the proof of the theorem.
\end{proof}

\section{\label{sec:Algebraic-approximation}Algebraic approximation for parameters
with dimension drop}

The purpose of this section is to establish Theorem \ref{thm:alg approx}.
The proof is similar to the proof of \cite[Theorem A.2]{rapaport20203maps},
which is a modification of the argument in \cite{BV-transcendent}.

In Section \ref{subsec:Convolution-factors}, we introduce notation
for convolution factors of $\mu_{\lambda}$ and state some basic properties.
In Section \ref{subsec:Non-saturation-of conv factors}, we establish
non-saturation of $\mu_{\lambda}$ and its convolution factors. In
Section \ref{subsec:Increasing-entropy-of conv}, we state a result
from \cite{BV-transcendent}, which relies on Theorem \ref{thm:effective ent inc result}
and plays a key role in the argument. In Section \ref{subsec:Diophantine-considerations},
we provide two necessary Diophantine statements. In Section \ref{subsec:Proof-of-Theorem alg approx},
we complete the proof of the theorem.

\subsection{\label{subsec:Convolution-factors}Convolution factors of $\mu_{\lambda}$
and basic properties}

Let $\{\xi_{n}\}_{n\in\mathbb{Z}}$ be a sequence of $\Lambda$-valued
i.i.d. random elements with $\mathbb{P}\{\xi_{0}=i\}=p_{i}$ for $i\in\Lambda$.
For a bounded subset $I$ of $\mathbb{R}_{>0}$, we denote by $\mu_{\lambda}^{I}\in\mathcal{M}(\mathbb{R}^{d})$
the law of random vector 
\[
\sum_{n\in\mathbb{Z}:\lambda_{1}^{n}\in I}a_{\xi_{n}}\lambda^{n},
\]
where recall from Section \ref{subsec:setup-and add notat} that $a_{i}:=(a_{i,1},\ldots,a_{i,d})$
for $i\in\Lambda$. Note that,
\begin{equation}
\mu_{\lambda}^{\lambda_{1}^{k}I}=S_{\lambda^{k}}\mu_{\lambda}^{I}\text{ for }k\in\mathbb{Z}.\label{eq:mu^=00007Blambda_1^k I=00007D=00003D}
\end{equation}
Also note that $\mu_{\lambda}^{(0,1]}=\mu_{\lambda}$, and so $\mu_{\lambda}^{I}$
is a convolution factor of $\mu_{\lambda}$ when $I\subset(0,1]$.
That is, there exists $\nu\in\mathcal{M}(\mathbb{R}^{d})$ such that
$\mu_{\lambda}=\nu*\mu_{\lambda}^{I}$.

For $n\in\mathbb{Z}_{>0}$ we have $\mu_{\lambda}^{(n)}=\mu_{\lambda}^{(\lambda_{1}^{n},1]}$,
where $\mu_{\lambda}^{(n)}$ is defined in Section \ref{subsec:setup-and add notat}.
By an argument similar to the one given in the proof of \cite[Lemma 6.4]{Rap_SA_diag},
\begin{equation}
H\left(\mu_{\lambda}^{(n')},\mathcal{E}_{n}\right)=H\left(\mu_{\lambda},\mathcal{E}_{n}\right)+O(1)\text{ for }n'\ge n\ge1.\label{eq:fini-sum-entr}
\end{equation}

For $\mu\in\mathcal{M}(\mathbb{R}^{d})$ and $t_{1},t_{2}\in\mathbb{R}_{>0}$
we write
\[
H\left(\mu;t_{1}\right):=H\left(\mu;\lambda^{-\frac{1}{\chi_{1}}\log t_{1}}\right)\text{ and }H\left(\mu;t_{1}\mid t_{2}\right):=H\left(\mu;t_{1}\right)-H\left(\mu;t_{2}\right),
\]
where recall that $\lambda^{t}=(\lambda_{1}^{t},...,\lambda_{d}^{t})$
for $t\in\mathbb{R}$. By Lemma \ref{lem:cond avg ent is pos =000026 ub}
\begin{equation}
H\left(\mu;t_{2}\right)\le H\left(\mu;t_{1}\right)\quad\text{ whenever }t_{2}\ge t_{1},\label{eq: posi-bou-condi-entr}
\end{equation}
and there exists a constant $C=C(\lambda)>0$ such that 
\begin{equation}
H\left(\mu;t_{1}\mid t_{2}\right)\le C\log\left(t_{2}/t_{1}\right)\quad\text{ whenever }t_{2}/t_{1}\ge2.\label{eq:ub on cond avg ent}
\end{equation}
From (\ref{eq: posi-bou-condi-entr}) we obtain,
\begin{equation}
H\left(\mu;t_{1}\mid t_{2}\right)\le H\left(\mu;t_{1}'\mid t_{2}'\right)\text{ for intervals }(t_{1},t_{2})\subset(t_{1}',t_{2}')\subset\mathbb{R}_{>0}.\label{eq:ineq for intervals}
\end{equation}
By (\ref{eq:mu^=00007Blambda_1^k I=00007D=00003D}) and (\ref{eq:scaling rel of avg ent}),
for every bounded $I\subset\mathbb{R}_{>0}$ and $k\in\mathbb{Z}$,
\begin{equation}
H\left(\mu_{\lambda}^{\lambda_{1}^{k}I};\lambda_{1}^{k}t_{1}\mid\lambda_{1}^{k}t_{2}\right)=H\left(\mu_{\lambda}^{I};t_{1}\mid t_{2}\right).\label{eq: scaling-inva}
\end{equation}

We shall need the following lemma, which extends \cite[equation (2.8)]{BV-transcendent}
to higher dimensions.
\begin{lem}
\label{lem:H()>=00003D1/2 H()}Let $I_{1}\subset I_{2}$ be bounded
subsets of $\mathbb{R}_{>0}$, and let $t_{1},t_{2}\in\mathbb{R}_{>0}$
be with $t_{2}/t_{1}\ge2$. Then,
\[
H\left(\mu_{\lambda}^{I_{2}};t_{1}\mid t_{2}\right)\ge\frac{1}{2}H\left(\mu^{I_{1}};t_{1}\mid t_{2}\right).
\]
\end{lem}

\begin{proof}
For $i=1,2$ and $1\le j\le d$ set $r_{i,j}:=\lambda_{j}^{-\frac{1}{\chi_{1}}\log t_{i}}$
and write $r_{i}:=(r_{i,1},...,r_{i,d})$. The condition $t_{2}/t_{1}\ge2$
implies that
\begin{equation}
N_{j}:=\lfloor r_{2,j}/r_{1,j}\rfloor\ge2\text{ for each }1\le j\le d.\label{eq:N_j>=00003D2}
\end{equation}
Setting
\[
r_{1}':=\left(r_{1,1}N_{1},\ldots,r_{1,d}N_{d}\right)\text{ and }r_{2}':=\left(r_{2,1}/N_{1},\ldots,r_{2,d}/N_{d}\right),
\]
it follows easily from (\ref{eq:N_j>=00003D2}) that $r_{1}\le r_{2}'\le r_{1}'\le r_{2}$.
Thus by Lemma \ref{lem:cond avg ent is pos =000026 ub},
\[
H\left(\mu_{\lambda}^{I_{1}};r_{1}\mid r_{1}'\right)+H\left(\mu_{\lambda}^{I_{1}};r_{2}'\mid r_{2}\right)\ge H\left(\mu^{I_{1}};r_{1}\mid r_{2}\right)=H\left(\mu^{I_{1}};t_{1}\mid t_{2}\right).
\]

The last inequality implies that there exist $r,r'\in\mathbb{R}_{>0}^{d}$,
with $r'/r\in\mathbb{Z}_{>0}^{d}$, so that
\[
\lambda^{-\frac{1}{\chi_{1}}\log t_{1}}\le r\le r'\le\lambda^{-\frac{1}{\chi_{1}}\log t_{2}}\text{ and }H\left(\mu_{\lambda}^{I_{1}};r\mid r'\right)\ge\frac{1}{2}H\left(\mu^{I_{1}};t_{1}\mid t_{2}\right).
\]
Hence by Lemmas \ref{lem:cond avg ent is pos =000026 ub} and \ref{lem:conv don't dec},
\[
H\left(\mu_{\lambda}^{I_{2}};t_{1}\mid t_{2}\right)\ge H\left(\mu_{\lambda}^{I_{2}};r\mid r'\right)\ge H\left(\mu_{\lambda}^{I_{1}};r\mid r'\right)\ge\frac{1}{2}H\left(\mu^{I_{1}};t_{1}\mid t_{2}\right),
\]
which completes the proof of the lemma.
\end{proof}

\subsection{\label{subsec:Non-saturation-of conv factors}Non saturation of convolution
factors}

The following proposition is the main result of this subsection.
\begin{prop}
\label{prop:non-satur}Suppose that $\dim\mu_{\lambda}<d$ and $\dim\pi_{J}\mu_{\lambda}=|J|$
for each proper subset $J$ of $[d]$. Then there exist $\epsilon=\epsilon(\lambda)>0$
and $M=M(\lambda)\ge1$ such that $\mu_{\lambda}^{I}$ is $(\epsilon,m)$-non-saturated
for all $m\ge M$ and bounded $I\subset\mathbb{R}_{>0}$.
\end{prop}

The proof requires some preparations. Recall the definition of $\kappa$
from section \ref{subsec:setup-and add notat}. The following lemma
is established in \cite[Lemma 4.1]{Rap_SA_diag}.
\begin{lem}
\label{lem:entr-dimens}Suppose that $\dim\pi_{[d-1]}\mu_{\lambda}=d-1$.
Then,
\[
\lim_{n\to\infty}\frac{1}{n}H(\mu_{\lambda},\mathcal{E}_{n})=\kappa.
\]
\end{lem}

We use this to prove the following statement.
\begin{lem}
\label{lem:before non-sat prop}Suppose that $\dim\mu_{\lambda}<d$
and $\dim\pi_{[d-1]}\mu_{\lambda}=d-1$. Then there exists $\epsilon=\epsilon(\lambda)>0$
so that for all $m\ge M(\lambda)\ge1$ and $n\ge0$, 
\begin{equation}
\frac{1}{m}H(\mu_{\lambda},\mathcal{E}_{n+m}\mid\mathcal{E}_{n})<\sum_{j=1}^{d}\chi_{j}-\epsilon.\label{eq:above-boun}
\end{equation}
\end{lem}

\begin{proof}
By the definition of $\kappa$ and since $\dim\mu_{\lambda}<d$ we
have $\sum_{j=1}^{d}\chi_{j}>\kappa$. Let $0<\epsilon<\frac{1}{3}\left(\sum_{j=1}^{d}\chi_{j}-\kappa\right)$,
let $m\in\mathbb{Z}_{>0}$ be large with respect to $\epsilon$, and
let $n\ge0$ be given. Assume by contradiction that (\ref{eq:above-boun})
fails to hold. Thus, from the concavity of conditional entropy, by
(\ref{eq:mu^=00007Blambda_1^k I=00007D=00003D}) and (\ref{eq:commens part by S_=00007Blambda^k=00007D}),
and since $\mu_{\lambda}=\mu_{\lambda}^{(0,1]}$, it follows that
for each $k\in\mathbb{Z}_{\ge0}$
\[
\begin{split}\frac{1}{m}H(\mu_{\lambda}, & \mathcal{E}_{n+(k+1)m}\mid\mathcal{E}_{n+km})\ge\frac{1}{m}H\left(\mu_{\lambda}^{(0,\lambda_{1}^{km}]},\mathcal{E}_{n+(k+1)m}\mid\mathcal{E}_{n+km}\right)+O\left(\frac{1}{m}\right)\\
 & =\frac{1}{m}H\left(\mu_{\lambda},\mathcal{E}_{n+m}\mid\mathcal{E}_{n}\right)+O\left(\frac{1}{m}\right)\ge\sum_{j=1}^{d}\chi_{j}-2\epsilon.
\end{split}
\]
Combining this with Lemma \ref{lem:entr-dimens}, we obtain
\[
\begin{split}\kappa & =\lim_{\ell\to\infty}\frac{1}{\ell m}H(\mu_{\lambda},\mathcal{E}_{\ell m})=\lim_{\ell\to\infty}\frac{1}{\ell m}H(\mu_{\lambda},\mathcal{E}_{\ell m+n}\mid\mathcal{E}_{n})\\
 & =\lim_{\ell\to\infty}\frac{1}{\ell}\sum_{k=0}^{\ell-1}\frac{1}{m}H(\mu_{\lambda},\mathcal{E}_{n+(k+1)m}\mid\mathcal{E}_{n+km})\ge\sum_{j=1}^{d}\chi_{j}-2\epsilon.
\end{split}
\]
This contradicts the choice of $\epsilon$, which completes the proof
of the lemma.
\end{proof}
We also need the following lemma, which is established in \cite[Lemma 1.14]{Rap_SA_diag}.
\begin{lem}
\label{lem:full proj imply foll cond ent}Let $J\subset[d]$ be with
$\dim\pi_{J}\mu_{\lambda}=|J|$. Then for all $\epsilon>0$, $m\ge M(\lambda,\epsilon)\ge1$
and $n\ge0$,
\[
\frac{1}{m}H\left(\mu_{\lambda},\pi_{J}^{-1}\mathcal{E}_{n+m}\mid\mathcal{E}_{n}\right)\ge\sum_{j\in J}\chi_{j}-\epsilon.
\]
\end{lem}

\begin{proof}[Proof of Proposition \ref{prop:non-satur}]
Let $\epsilon>0$ be as obtained in Lemma \ref{lem:before non-sat prop},
let $m\in\mathbb{Z}_{>0}$ be large with respect to $\epsilon$, and
let $1\le j\le d$ and $n\in\mathbb{Z}_{\ge0}$ be given. Since $\dim\pi_{[d]\setminus\{j\}}\mu_{\lambda}=d-1$
and by Lemma \ref{lem:full proj imply foll cond ent},
\[
\frac{1}{m}H\left(\mu_{\lambda},\pi_{[d]\setminus\{j\}}^{-1}\mathcal{E}_{n+m}\mid\mathcal{E}_{n}\right)\ge\sum_{l=1}^{d}\chi_{l}-\chi_{j}-\epsilon/2.
\]
From this and by Lemma \ref{lem:before non-sat prop},
\begin{multline*}
\frac{1}{m}H\left(\mu_{\lambda},\mathcal{E}_{n+m}\mid\mathcal{E}_{n}\vee\pi_{[d]\setminus\{j\}}^{-1}\mathcal{E}_{n+m}\right)\\
=\frac{1}{m}H\left(\mu_{\lambda},\mathcal{E}_{n+m}\mid\mathcal{E}_{n}\right)-\frac{1}{m}H\left(\mu_{\lambda},\pi_{[d]\setminus\{j\}}^{-1}\mathcal{E}_{n+m}\mid\mathcal{E}_{n}\right)\le\chi_{j}-\epsilon/2.
\end{multline*}
From the last inequality and by concavity,
\[
\frac{1}{m}H\left(\mu_{\lambda}^{I},\mathcal{E}_{n+m}\mid\mathcal{E}_{n}\vee\pi_{[d]\setminus\{j\}}^{-1}\mathcal{E}_{n+m}\right)\le\chi_{j}-\epsilon/2\text{ for all }I\subset(0,1]\text{ and }n\ge0.
\]

Note that given a bounded $I\subset\mathbb{R}_{>0}$, there exist
$k\ge1$ and $I'\subset(0,1]$ so that $I=\lambda_{1}^{-k}I'$. The
proposition now follows from this, from the last formula, and by (\ref{eq:mu^=00007Blambda_1^k I=00007D=00003D})
and (\ref{eq:commens part by S_=00007Blambda^k=00007D}).
\end{proof}
We end this subsection with the following corollary, which follows
directly from Theorem \ref{thm:effective ent inc result} and Proposition
\ref{prop:non-satur}.
\begin{cor}
\label{cor:ent inc for conv factors}Suppose that $\dim\mu_{\lambda}<d$
and $\dim\pi_{J}\mu_{\lambda}=|J|$ for each proper subset $J$ of
$[d]$. Then there exists $C=C(\lambda)>1$ such that the following
holds. Let $\nu\in\mathcal{M}(\mathbb{R}^{d})$, $0<\beta<1/2$ and
$t_{2}>t_{1}>0$ be with $\frac{1}{t_{2}-t_{1}}H(\nu;2^{-t_{2}}\mid2^{-t_{1}})>\beta$.
Then for all bounded $I\subset\mathbb{R}_{>0}$,
\[
H\left(\nu*\mu_{\lambda}^{I};2^{-t_{2}}\mid2^{-t_{1}}\right)\ge H\left(\mu_{\lambda}^{I};2^{-t_{2}}\mid2^{-t_{1}}\right)+C^{-1}\beta\left(\log\beta^{-1}\right)^{-1}(t_{2}-t_{1})-C.
\]
\end{cor}

\subsection{\label{subsec:Increasing-entropy-of conv}Increasing entropy of convolutions}

The following proposition plays a key role in the proof of Theorem
\ref{thm:alg approx}. We have Corollary \ref{cor:ent inc for conv factors},
properties (\ref{eq:ub on cond avg ent}), (\ref{eq:ineq for intervals}),
and (\ref{eq: scaling-inva}), and Lemma \ref{lem:H()>=00003D1/2 H()}
at our disposal. With these, the proof of the proposition is almost
identical to the proof of \cite[Proposition 30]{BV-transcendent},
which deals with the case $d=1$. Therefore, the proof is omitted.
\begin{prop}
\label{pro:A.3}Suppose that $\dim\mu_{\lambda}<d$ and $\dim\pi_{J}\mu_{\lambda}=|J|$
for each proper subset $J$ of $[d]$. Then for all $\epsilon>0$
there exists $C=C(\lambda,\epsilon)>1$ such that the following holds.
Let $N\ge1$, $\{n_{j}\}_{j=1}^{N}\subset\mathbb{Z}_{>0}$ and $\{K_{j}\}_{j=1}^{N}\subset[10,\infty)$
be given. Suppose that $\lambda_{1}^{-n_{1}}\ge\max\{2,\lambda_{1}^{-2}\}$
and,
\begin{enumerate}
\item $n_{j+1}\ge K_{j}n_{j}$ for all $1\le j<N$;
\item $H\left(\mu_{\lambda}^{(n_{j})};\lambda_{1}^{K_{j}n_{j}}\mid\lambda_{1}^{10n_{j}}\right)\ge\epsilon n_{j}$
for all $1\le j\le N$;
\item $n_{j}\ge C(\log K_{j})^{2}$ for all $1\le j\le N$.
\end{enumerate}
Then, 
\[
\sum_{j=1}^{N}\frac{1}{\log K_{j}\log\log K_{j}}\le C\left(1+\frac{1}{n_{1}}\sum_{j=1}^{N}\log K_{j}\right).
\]
\end{prop}

\begin{rem}
In \cite[Proposition 30]{BV-transcendent}, in the case $d=1$, it
is assumed that there exists $\alpha>0$ such that $H\left(\mu_{\lambda};t\mid2t\right)\le1-\alpha$
for all $t>0$. This assumption is needed only to apply \cite[Theorem 8]{BV-transcendent},
a version of Theorem \ref{thm:effective ent inc result} for the case
$d=1$, to convolution factors of $\mu_{\lambda}$. Since we already
have Corollary \ref{cor:ent inc for conv factors}, this assumption
is not required here.
\end{rem}

\subsection{\label{subsec:Diophantine-considerations}Diophantine considerations}

The following two propositions will be required during the proof of
Theorem \ref{thm:alg approx}. Recall the definition of $L_{0}$ from
(\ref{eq:def of L_0}).
\begin{prop}
\label{pro:A.5}There exists $C=C(\lambda)>1$ such that the following
holds for all $n\ge N(\lambda)\ge1$. Let $0<t<n^{-Cn}$, and suppose
that $\frac{1}{n}H\left(\mu_{\lambda}^{(n)};t\right)<H(p)$. Then
there exists $(\eta_{1},...,\eta_{d})=\eta\in\Omega$ so that $\eta_{j}$
is a root of a nonzero polynomial in $\mathcal{P}_{L_{0}}^{(n)}$
for each $1\le j\le d$, $|\lambda-\eta|<t^{1/C}$, and 
\[
H(\mu_{\eta}^{(n)})\le H\left(\mu_{\lambda}^{(n)};t\right).
\]
\end{prop}

\begin{proof}
Let $C>1$ be large with respect to $\lambda$, let $n\in\mathbb{Z}_{>0}$
be with $C\ll n$, and let $0<t<n^{-Cn}$ be such that $\frac{1}{n}H\left(\mu_{\lambda}^{(n)};t\right)<H(p)$.
There exists $x\in[0,1)^{d}$ so that
\begin{equation}
H\left(\left\lfloor \lambda^{\frac{1}{\chi_{1}}\log t}\sum_{k=0}^{n-1}a_{\xi_{k}}\lambda^{k}+x\right\rfloor \right)\le H\left(\mu_{\lambda}^{(n)};t\right)<nH(p),\label{eq: A.4.1}
\end{equation}
where $\{\xi_{k}\}_{k\in\mathbb{Z}}$ is as in Section \ref{subsec:Convolution-factors}.

For each $1\leq j\leq d$ set
\[
r_{j}=\lambda_{j}^{-\frac{1}{\chi_{1}}\log t}\text{ and }D_{j}:=\left\{ a_{i_{1},j}-a_{i_{2},j}:i_{1},i_{2}\in\Lambda\right\} ,
\]
and let $\mathcal{A}_{j}$ be the set of all nonzero polynomials $P(X)=\sum_{k=0}^{n-1}d_{k}X^{k}\in\mathbb{Z}[X]$
with $d_{0},...,d_{k-1}\in D_{j}$ and $|P(\lambda_{j})|\leq r_{j}$.
Note that $\mathcal{A}_{j}\subset\mathcal{P}_{L_{0}}^{(n)}$, and
that by (\ref{eq: A.4.1}) it follows that $\mathcal{A}_{j}$ is nonempty.
From this, since $r_{j}\le t$, and by an argument appearing in the
proof of \cite[Proposition A.5]{rapaport20203maps}, it follows that
there exists $\eta_{j}\in(0,1)$ such that $|\lambda_{j}-\eta_{j}|<r_{j}^{1/C}/d$
and $P(\eta_{j})=0$ for each $P\in\mathcal{A}_{j}$.

Setting $\eta:=(\eta_{1},...,\eta_{d})$ we have $|\lambda-\eta|<t^{1/C}<n^{-n}$.
By assuming that $n$ is sufficiently large and since $\lambda\in\Omega$,
we obtain $\eta\in\Omega$. Moreover, from (\ref{eq: A.4.1}) and
by the definition of the sets $\mathcal{A}_{1},...,\mathcal{A}_{d}$,
we get $H(\mu_{\eta}^{(n)})\le H\left(\mu_{\lambda}^{(n)};r\right)$.
This completes the proof of the proposition.
\end{proof}
The proof of the following statement, which relies on Proposition
\ref{pro:A.5}, is similar to the proof of \cite[Proposition A.7]{rapaport20203maps}
and is therefore omitted.
\begin{prop}
\label{pro:A.7}There exists $C=C(\lambda)>1$ such that the following
holds for all $n\ge N(\lambda)\ge1$. Suppose that there exists $(\eta_{1},...,\eta_{d})=\eta\in\Omega$
such that $|\lambda-\eta|<n^{-Cn}$ and $\eta_{j}$ is a root of a
nonzero polynomial in $\mathcal{P}_{L_{0}}^{(n)}$ for each $1\leq j\leq d$.
Then $\frac{1}{n}H(\mu_{\lambda}^{(n)};t)=H(p)$ for all $0<t\le|\lambda-\eta|^{C}$.
\end{prop}

\subsection{\label{subsec:Proof-of-Theorem alg approx}Proof of Theorem \ref{thm:alg approx}}

We shall need the following statement, which follows directly from
\cite[Theorem 6.3]{Rap_SA_diag} and (\ref{eq:two defs of ent are equiv}).
\begin{thm}
\label{thm:cond ent tends to 0}Suppose that $\dim\mu_{\lambda}<d$
and that $\dim\pi_{J}\mu_{\lambda}=|J|$ for each proper subset $J$
of $[d]$. Then for any $q>1$,
\[
\lim_{n\to\infty}\frac{1}{n}H\left(\mu_{\lambda}^{(n)};\lambda_{1}^{qn}\mid\lambda_{1}^{n}\right)=\lim_{n\to\infty}\frac{1}{n}H\left(\mu_{\lambda}^{(n)},\mathcal{E}_{qn}\mid\mathcal{E}_{n}\right)=0.
\]
\end{thm}

Next, we prove Theorem \ref{thm:alg approx}, which is the following
statement.
\begin{thm*}
Suppose that $\dim\mu_{\lambda}<\gamma$, $\dim\pi_{J}\mu_{\lambda}=|J|$
for each proper subset $J$ of $[d]$, and $\lambda_{j_{0}}$ is transcendental
for some $1\le j_{0}\le d$. Then for every $\epsilon>0$ and $N\ge1$
there exist $n\ge N$ and $(\eta_{1},\ldots,\eta_{d})=\eta\in\Omega$
such that,
\begin{enumerate}
\item for each $1\le j\le d$ there exists $0\neq P_{j}\in\mathcal{P}_{L_{0}}^{(n)}$
with $P_{j}(\eta_{j})=0$;
\item $h_{RW}(\Phi^{\eta},p)<\kappa+\epsilon$;
\item $|\lambda-\eta|\le\exp\left(-n^{1/\epsilon}\right)$.
\end{enumerate}
\end{thm*}
\begin{proof}
By $\text{dim }\mu_{\lambda}<\gamma$ and (\ref{eq:ub for LY-dim})
with $m=d-1$, we have $\kappa<H(p)$. Let
\[
0<\epsilon<\frac{1}{3}\min\left\{ 1,H(p)-\kappa\right\} 
\]
be given. For each $n\ge1$ let $E^{(n)}$ be the set of all $(\eta_{1},...,\eta_{d})=\eta\in\Omega$
such that $h_{RW}(\Phi^{\eta},p)<\kappa+2\epsilon$ and for each $1\le j\le d$
there exists $0\neq P_{j}\in\mathcal{P}_{L_{0}}^{(n)}$ with $P_{j}(\eta_{j})=0$.
Assume by contradiction that,
\begin{equation}
|\lambda-\eta|>\exp\left(-n^{1/\epsilon}\right)\text{ for all }n\ge\epsilon^{-1}\text{ and }\eta\in E^{(n)}.\label{eq:all eta in E^(n)}
\end{equation}

Let $C>1$ and $n_{0}\in\mathbb{Z}_{\ge1}$ be with $\epsilon^{-1}\ll C\ll n_{0}$,
and suppose also that $C$ is large with respect to $\lambda$. We
next define by induction a sequence $n_{0}<n_{1}<n_{2}<\ldots$ of
positive integers so that for each $j\ge0$
\begin{equation}
\left\lceil \frac{Cn_{j}\log n_{j}}{\log1/\lambda_{1}}\right\rceil \le n_{j+1}<n_{j}^{2/\epsilon},\label{eq:comp of n_j =000026 n_j+1}
\end{equation}
and
\begin{equation}
H\left(\mu_{\lambda}^{(n_{j+1})};n_{j+1}^{-Cn_{j+1}}\mid\lambda_{1}^{10n_{j+1}}\right)\ge\epsilon n_{j+1}.\label{eq:A.5.3}
\end{equation}

Let $j\geq0$, assume $n_{j}$ has been chosen, and set $q=\left\lceil \frac{Cn_{j}\log n_{j}}{\log1/\lambda_{1}}\right\rceil $.
Since $\lambda_{1}^{-1},\epsilon^{-1}\ll C\ll n_{0}$, we may assume
that $q<n_{j}^{2}$.

Suppose first that
\begin{equation}
H\left(\mu_{\lambda}^{(q)};q^{-Cq}\right)\ge q\left(\kappa+2\epsilon\right),\label{eq:first scenario}
\end{equation}
in which case we set $n_{j+1}=q$. From Lemma \ref{lem:entr-dimens},
from (\ref{eq:fini-sum-entr}) and (\ref{eq:two defs of ent are equiv}),
and by Theorem \ref{thm:cond ent tends to 0},
\begin{equation}
H\left(\mu_{\lambda}^{(n)};\lambda_{1}^{10n}\right)\le n\left(\kappa+\epsilon\right)\text{ for all }n\ge q.\label{eq:ent <=00003D all n =00005Cgeq}
\end{equation}
Hence, by (\ref{eq:first scenario}) it follows that (\ref{eq:A.5.3})
is satisfied. Additionally, note that $n_{j+1}=q<n_{j}^{2}<n_{j}^{2/\epsilon}$
and so (\ref{eq:comp of n_j =000026 n_j+1}) also holds.

Next suppose that $H\left(\mu_{\lambda}^{(q)};q^{-Cq}\right)<q\left(\kappa+2\epsilon\right)$,
which implies $\frac{1}{q}H\left(\mu_{\lambda}^{(q)};q^{-Cq}\right)<H(p)$.
By Proposition \ref{pro:A.5}, there exists $(\eta_{1},...,\eta_{d})=\eta\in\Omega$
so that $\eta_{j}$ is a root of a nonzero polynomial in $\mathcal{P}_{L_{0}}^{(q)}$
for each $1\le j\le d$, $|\lambda-\eta|<q^{-C^{1/2}q}$, and
\[
H(\mu_{\eta}^{(q)})\le H\left(\mu_{\lambda}^{(q)};q^{-Cq}\right)<q\left(\kappa+2\epsilon\right).
\]
Together with (\ref{eq:h_RW =00003D inf}), the last inequality implies
$h_{RW}(\Phi^{\eta},p)\le\kappa+2\epsilon$, and so $\eta\in E^{(q)}$.

By assumption $\lambda_{j_{0}}$ is transcendental for some $1\le j_{0}\le d$,
and so $\lambda\ne\eta$. We choose $n_{j+1}$ to be the largest $n\in\mathbb{Z}_{>0}$
so that $|\lambda-\eta|<n^{-C^{1/2}n}$. Since $|\lambda-\eta|<q^{-C^{1/2}q}$
we have $n_{j+1}\ge q$. From
\[
(n_{j+1}+1)^{-C^{1/2}(n_{j+1}+1)}\le|\lambda-\eta|<n_{j+1}^{-C^{1/2}n_{j+1}}
\]
and by Proposition \ref{pro:A.7}, we obtain 
\[
H\left(\mu_{\lambda}^{(n_{j+1})};(n_{j+1}+1)^{-C(n_{j+1}+1)}\right)=n_{j+1}H(p)>n_{j+1}\left(\kappa+3\epsilon\right).
\]
From this, from (\ref{eq:ent <=00003D all n =00005Cgeq}), and by
Lemma \ref{lem:cond avg ent is pos =000026 ub}, we get that (\ref{eq:A.5.3})
also holds in the present case.

From (\ref{eq:all eta in E^(n)}) and since $\eta\in E^{(q)}$, we
get $|\lambda-\eta|>\exp\left(-q^{1/\epsilon}\right)$. Also recall
that $q<n_{j}^{2}$. By combining these facts together with $|\lambda-\eta|<n_{j+1}^{-C^{1/2}n_{j+1}}$,
we obtain
\begin{equation}
n_{j+1}<\log\left(n_{j+1}^{C^{1/2}n_{j+1}}\right)<-\log|\lambda-\eta|<q^{1/\epsilon}<n_{j}^{2/\epsilon}.\label{eq:A.5.7}
\end{equation}
Thus, (\ref{eq:comp of n_j =000026 n_j+1}) is satisfied once more,
completing the inductive construction of $\{n_{j}\}_{j\ge0}$.

We next aim to apply Proposition \ref{pro:A.3} in order to derive
the desired contradiction. Set
\[
j_{0}:=\lceil\log^{(2)}n_{0}\rceil\text{ and }N:=\left\lceil \exp^{(2)}\left(\log^{(2)}(j_{0}+1)+C^{2}\right)\right\rceil ,
\]
where $\log^{(2)}$ stands for the composition of the $\log$ function
with itself, and similarly for $\exp^{(2)}$. Additionally, for each
$1\le j\le N$ set $K_{j}:=\frac{C\log n_{j}}{\log\lambda_{1}^{-1}}$.
By (\ref{eq:A.5.3}),
\[
H\left(\mu_{\lambda}^{(n_{j})};\lambda_{1}^{K_{j}n_{j}}\mid\lambda_{1}^{10n_{j}}\right)\ge\epsilon n_{j}.
\]
Moreover, by (\ref{eq:comp of n_j =000026 n_j+1}) and since $C\ll n_{0}\le n_{j}$,
we also have $\lambda_{1}^{-n_{1}}\ge\max\{2,\lambda_{1}^{-2}\}$,
$n_{j}\ge C(\log K_{j})^{2}$ and 
\[
n_{j+1}\ge\left\lceil \frac{Cn_{j}\log n_{j}}{\log\lambda_{1}^{-1}}\right\rceil \ge K_{j}n_{j}.
\]
Thus, by Proposition \ref{pro:A.3},
\begin{equation}
\sum_{j=1}^{N}\frac{1}{\log K_{j}\log\log K_{j}}\le C\left(1+\frac{1}{n_{1}}\sum_{j=1}^{N}\log K_{j}\right).\label{eq:A.5.5}
\end{equation}

Note that by applying (\ref{eq:comp of n_j =000026 n_j+1}) successively,
we get $n_{j}\le n_{0}^{2^{j}\epsilon^{-j}}$ for each $j\ge1$. From
this, since $\epsilon^{-1}\ll C\ll n_{0}$, and by the estimates carried
out at the end of \cite[Section 7.3, proof of Theorem 3.1]{rapaport20203maps},
it follows that
\[
\sum_{j=1}^{N}\frac{1}{\log K_{j}\log\log K_{j}}\ge\frac{\epsilon}{12}C^{2}\:\text{ and }\:\frac{1}{n_{1}}\sum_{j=1}^{N}\log K_{j}\le1.
\]
This together with (\ref{eq:A.5.5}) contradicts $\epsilon^{-1}\ll C$,
which completes the proof of the theorem.
\end{proof}

\section{\label{sec:Proof-of-main result}Proof of main result}

In this section, we prove Theorem \ref{thm:main gen thm}. Sections
\ref{subsec:Mahler-measure-and}, \ref{subsec:A-consequence-of dim drop},
and \ref{subsec:Number-theoretic-results} contain necessary preparations.
The proof of the theorem, which is an extension of the argument given
in \cite{Var-Bernoulli} and its modification found in \cite[Appendix A]{rapaport20203maps},
is carried out in Section \ref{subsec:Proof-of-main-result}.

\subsection{\label{subsec:Mahler-measure-and}Mahler measure and a lower bound
on random walk entropy}

Let $\alpha\in\mathbb{C}$ be an algebraic number with minimal polynomial
\[
P(X)=b(X-\alpha_{1})...(X-\alpha_{n})\in\mathbb{Z}[X],
\]
so that $b$ is the leading coefficient and $\alpha_{1},...,\alpha_{n}$
are the roots (including $\alpha$). The Mahler measure of $\alpha$
is defined by,
\[
M(\alpha):=|b|\prod_{k=1}^{n}\max\left\{ 1,|\alpha_{k}|\right\} .
\]

The following theorem follows directly from \cite[Proposition 13]{BV-entropy}.
For the details, we refer to the proof of \cite[Theorem 9]{Var-Bernoulli}.
For $\eta\in\Omega$ and $1\le j\le d$, recall the notation $\Phi_{j}^{\eta}$
from Section \ref{subsec:setup-and add notat}.
\begin{thm}
\label{thm:lb on RW-ent for large Mahler}Let $1\le j_{0}\le d$ and
$h\in(0,H(p))$ be given. Then there exists $M>1$, depending only
on $h$, $\{a_{i,j_{0}}\}_{i\in\Lambda}$ and $p$, so that $h_{RW}(\Phi_{j_{0}}^{\eta},p)>h$
for all $(\eta_{1},...,\eta_{d})=\eta\in\Omega$ such that $\eta_{j_{0}}$
is algebraic with $M(\eta_{j_{0}})\ge M$.
\end{thm}

\subsection{\label{subsec:A-consequence-of dim drop}A consequence of dimension
drop}
\begin{prop}
\label{prop:exists polynomials}Suppose that $\dim\mu_{\lambda}<\gamma$
and $\dim\pi_{J}\mu_{\lambda}=|J|$ for each proper subset $J$ of
$[d]$. Then for every $\epsilon>0$ there exists $N\ge1$, such that
for every $n\ge N$ and $1\le j\le d$ there exists $0\ne P_{j}\in\mathcal{P}_{L_{0}}^{(n)}$
with $|P_{j}(\lambda_{j})|<\epsilon^{n}$.
\end{prop}

\begin{proof}
Assume by contradiction that the proposition is false. Then there
exist $\epsilon>0$, $1\le j\le d$, and an increasing sequence $\{n_{k}\}_{k\ge1}\subset\mathbb{Z}_{>0}$,
such that $|P(\lambda_{j})|\ge\epsilon^{n_{k}}$ for all $0\ne P\in\mathcal{P}_{L_{0}}^{(n_{k})}$.
From this it follows easily that there exists $q>1$ so that
\[
\frac{1}{n_{k}}\left(\mu_{\lambda}^{(n_{k})},\mathcal{E}_{qn_{k}}\right)=H(p)\text{ for each }k\ge1,
\]
where $\mu_{\lambda}^{(n_{k})}$ is defined in Section \ref{subsec:setup-and add notat}.
On the other hand, by Lemma \ref{lem:entr-dimens}, (\ref{eq:fini-sum-entr}),
and Theorem \ref{thm:cond ent tends to 0},
\[
\lim_{k\to\infty}\frac{1}{n_{k}}H\left(\mu_{\lambda}^{(n_{k})},\mathcal{E}_{qn_{k}}\right)=\kappa.
\]
Hence $\kappa=H(p)$, which, by (\ref{eq:ub for LY-dim}) with $j=d-1$,
implies that $\dim\mu_{\lambda}=\gamma$. But this contradicts our
assumption, thus completing the proof of the proposition.
\end{proof}

\subsection{\label{subsec:Number-theoretic-results}Number theoretic results}

We shall need the following two lemmas from \cite{rapaport20203maps}.
\begin{lem}[{\cite[Lemma 4.2]{rapaport20203maps}}]
\label{lem:There-is-function r}There is a function $r:\mathbb{Z}_{>0}\rightarrow(0,1)$
such that $\underset{k\rightarrow\infty}{\lim}r(k)=1$ and the following
holds. Let $l,n\ge1$ and $0\ne P\in\mathcal{P}_{l}^{(n)}$ be given.
Then there are at most $k\left(1+\frac{\log l}{\log(k+1)}\right)$
nonzero roots of $P$ of absolute value less than $r(k)$.
\end{lem}

The next lemma and its proof were communicated to the authors of \cite{rapaport20203maps}
by Vesselin Dimitrov.
\begin{lem}[{\cite[Lemma 4.6]{rapaport20203maps}}]
\label{lem:vesselin lemma}Let $\xi,\eta\in[0,1]$ and $n,n',l,k\in\mathbb{Z}_{>0}$.
Let $0\ne P\in\mathcal{P}_{l}^{(n')}$. Let $\alpha$ be a number
that satisfies
\[
\log\alpha>\frac{(n(k+1)+(k+2))\log n'+(n+1)\log l+\log2}{n'}.
\]
Assume that $\eta\ne\xi$ and that $\eta$ is algebraic of degree
at most $n$. Assume that
\[
\left(\alpha M(\eta)\right)^{n'/k}\left|P(\xi)\right|^{1/k}\le|\xi-\eta|\le\left(\alpha M(\eta)\right)^{-n'}.
\]
Then $\eta$ is a zero of $P$ of order at least $k$.
\end{lem}

\subsection{\label{subsec:Proof-of-main-result}Proof of Theorem \ref{thm:main gen thm}}

Recall that $K_{\Phi^{\lambda}}$ denotes the attractor of $\Phi^{\lambda}$.
Note that, by the definitions of the affinity dimension (see \cite{falconer1988hausdorff})
and Lyapunov dimension (see Section \ref{subsec:setup-and add notat}),
in order to show that $\dim_{H}K_{\Phi^{\lambda}}=\min\left\{ d,\dim_{A}\Phi^{\lambda}\right\} $,
it suffices to prove $\dim\mu_{\lambda}=\min\left\{ d,\dim_{L}(\Phi^{\lambda},p)\right\} $
for $p$ with equal weights. Moreover, by rescaling the IFS if necessary,
it suffices to prove Theorem \ref{thm:main gen thm} in the case of
integral translations. Thus we need to establish the following statement.
\begin{thm*}
Suppose that $\Phi_{j}^{\lambda}$ has no exact overlaps for each
$1\le j\le d$. Then $\dim\mu_{\lambda}=\gamma$.
\end{thm*}
\begin{proof}
The proof is carried out by induction on $d$. Thus, assume that the
theorem holds whenever the dimension of the ambient space is strictly
less than $d$. Since $\dim_{L}(\Phi^{\lambda},p)$ is always an upper
bound for $\dim\mu_{\lambda}$, we only need to show that $\dim\mu_{\lambda}\ge\gamma$.

As shown in \cite[Section 6.4, proof of Theorem 1.7]{Rap_SA_diag},
by the induction hypothesis and the Ledrappier--Young formula, it
follows that $\dim\mu_{\lambda}=\dim_{L}(\Phi^{\lambda},p)$ whenever
$\dim\pi_{J}\mu_{\lambda}<|J|$ for some proper subset $J$ of $[d]$.
Thus, we may assume that $\dim\pi_{J}\mu_{\lambda}=|J|$ for each
$J\subsetneq[d]$. Moreover, by \cite[Theorem 1.7]{Rap_SA_diag} we
may suppose that $\lambda_{j_{0}}$ is transcendental for some $1\le j_{0}\le d$.

Assume by contradiction that $\dim\mu_{\lambda}<\gamma$. From this
and by (\ref{eq:ub for LY-dim}) with $j=d-1$, it follows that $\kappa<H(p)$.
Let 
\[
0<\epsilon<\frac{H(p)-\kappa}{2},
\]
let $M\in\mathbb{Z}_{>0}$ be large with respect to $\{a_{i,j}\}$,
$p$, $\lambda$ and $\epsilon$, and let $q_{0}\in\mathbb{Z}_{>0}$
be with $M\ll q_{0}$.

By Theorem \ref{thm:alg approx}, there exist an integer $q\ge q_{0}$
and $(\eta_{1},...,\eta_{d})=\eta\in\Omega$ such that $\eta_{j}$
is algebraic with $\deg\eta_{j}<q$ for each $1\le j\le d$,
\begin{equation}
h_{RW}(\Phi^{\eta},p)<\kappa+\epsilon<H(p)-\epsilon,\label{eq:ub of rw ent of Phi^eta}
\end{equation}
and $|\lambda-\eta|<2^{-q^{2}}$. From (\ref{eq:ub of rw ent of Phi^eta})
it clearly follows that $h_{RW}(\Phi_{j_{0}}^{\eta},p)<H(p)-\epsilon$.
Thus by Theorem \ref{thm:lb on RW-ent for large Mahler} we may assume
that $M(\eta_{j_{0}})<M$.

Since $\lambda_{j_{0}}$ is transcendental, we have $\lambda_{j_{0}}\ne\eta_{j_{0}}$.
Let $n\ge1$ be with, 
\[
(2M)^{-n-1}\le\left|\lambda_{j_{0}}-\eta_{j_{0}}\right|<(2M)^{-n}.
\]
We have, 
\[
(n+1)\log(2M)\ge-\log\left|\lambda_{j_{0}}-\eta_{j_{0}}\right|>q^{2}.
\]
Hence, we may assume that 
\begin{equation}
\frac{1}{n}\left(\left(q(M+1)+(M+2)\right)\log n+(q+1)\log L_{0}+1\right)<1,\label{eq:ineq for ves lem}
\end{equation}
where $L_{0}$ is defined in (\ref{eq:def of L_0}).

By Proposition \ref{prop:exists polynomials} and since $n$ is arbitrarily
large with respect to $M$, there exists $0\ne P\in\mathcal{P}_{L_{0}}^{(n)}$
such that $|P(\lambda_{j_{0}})|\le(2M)^{-3Mn}$, which gives, 
\[
\left|\lambda_{j_{0}}-\eta_{j_{0}}\right|\ge(2M)^{-n-1}\ge(2M)^{n}|P(\lambda_{j_{0}})|^{1/M}.
\]
From this, $\left|\lambda_{j_{0}}-\eta_{j_{0}}\right|<(2M)^{-n}$,
$\deg\eta_{j_{0}}<q$, the inequality (\ref{eq:ineq for ves lem}),
and Lemma \ref{lem:vesselin lemma}, it follows that $\eta_{j_{0}}$
is a zero of $P$ of order at least $M$. Now, by assuming $\eta_{j_{0}}<(1+\lambda_{j_{0}})/2$
and that $M$ is sufficiently large with respect to $L_{0}$ and $\lambda_{j_{0}}$,
we get a contradiction with Lemma \ref{lem:There-is-function r}.
This completes the proof of the theorem. 
\end{proof}
\bibliographystyle{plain}
\bibliography{bibfile}

$\newline$\textsc{Ariel Rapaport, Department of Mathematics, Technion, Haifa, Israel}$\newline$\textit{E-mail: }
\texttt{arapaport@technion.ac.il}

$\newline$\textsc{Haojie Ren, Department of Mathematics, Technion, Haifa, Israel}$\newline$\textit{E-mail: }
\texttt{hjren@campus.technion.ac.il}
\end{document}